\documentclass[a4paper,10pt]{amsart}

\usepackage{amsmath,amsfonts,amssymb,mathrsfs}
\usepackage{mathtools}
\usepackage{graphicx}
\usepackage{comment}
\usepackage{xfrac}
\usepackage{bbm}
\usepackage{epstopdf}
\usepackage{xcolor}
\usepackage[colorlinks,pdfpagelabels,pdfstartview = FitH,bookmarksopen = true,bookmarksnumbered = true,urlcolor=url,linkcolor = formula,plainpages = false,hypertexnames = false,citecolor = citation] {hyperref}
\usepackage{comment}
\usepackage{dsfont}
\usepackage{fancyhdr}
\usepackage[english]{babel}
\usepackage{lmodern}
\usepackage{textcomp}
\usepackage{mathscinet}
\providecommand{\keywords}[1]{\textbf{\sc{keywords:}} #1}
\usepackage{cancel}
\usepackage{appendix}

\definecolor{magenta}{rgb}{1.0, 0.13, 0.32}
\definecolor{citation}{rgb}{0.2,0.58,0.2} 
\definecolor{formula}{rgb}{0.1,0.2,0.6}
\definecolor{url}{rgb}{0.3,0,0.5} 

\newtheorem{lemma}{Lemma}[section]
\newtheorem{theorem}{Theorem}[section]
\newtheorem{defn}{Definition}[section]
\newtheorem{corol}{Corollary}[section]

\newtheorem{ex}{Example}[section]
\newtheorem{rem}{Remark}[section]

\numberwithin{equation}{section}

\def \dd {{\rm d}}

\newcommand{\vs}{\vspace{1mm}}
\newcommand{\R}{\mathds{R}}

\newcommand{\N}{\mathds{N}}

\newcommand{\OO}{\mathds{O}}
\newcommand{\I}{\mathds{I}}

\newcommand{\Tl}{{\rm Tail}}

\newcommand{\W}{\mathcal{W}}
\newcommand\ap{``}

\def \rnn {{\mathds {R}}^{2n+1}}

\def \L {\mathscr{L}}
\def \Lc {\mathcal{L}_K}
\def \K {\mathscr{K}}

\def \Ac {\mathcal{A}}

\def \a {{\alpha}}

\def \d {{\delta}}

\def \eps {{\varepsilon}}
\def \epsilon {{\varepsilon}}

\def \k {{\kappa}}
\def \l {{\lambda}}
\def \r {{\rho}}
\def \s {{\sigma}}
\def \t {{\tau}}

\def \x {{\xi}}

\def \z {{\zeta}}

\def \phi {{\varphi}}

\def \G {{\Gamma}}
\def \O {{\Omega}}
\def \Ob {{\overline{\O}}}
\def \Oo {{\Omega_{0}}}
\def \Om {{\Omega_{m_0}}}
\def \Onm {{\Omega_{N-m_0+1}}}
\def \Ov {{\Omega_v}}
\def \Ox {{\Omega_x}}
\def \D {{\Delta}}

\def \div {{\text{\rm div}}}
\def \loc {{\text{\rm loc}}}
\def \trace {{\text{\rm tr}}}

\def \diag {{\text{\rm diag}}}
\def \meas {{\text{\rm meas}}}
\def\p{\partial}
\def \tilde {\widetilde}

\def \supp {{\text{\rm supp}}}

\def \Q {{\mathcal{Q}}}

\def \D {{\mathcal{D}}}

\usepackage{mathtools}

\allowdisplaybreaks
\makeatletter
\DeclareRobustCommand*{\bfseries}{%
	\not@math@alphabet\bfseries\mathbf
	\fontseries\bfdefault\selectfont
	\boldmath
}

\def\Xint#1{\mathchoice
	{\XXint\displaystyle\textstyle{#1}}%
	{\XXint\textstyle\scriptstyle{#1}}%
	{\XXint\scriptstyle\scriptscriptstyle{#1}}%
	{\XXint\scriptscriptstyle\scriptscriptstyle{#1}} %
	\!\int}
\def\XXint#1#2#3{{\setbox0=\hbox{$#1{#2#3}{\int}$}
		\vcenter{\hbox{$#2#3$}}\kern-.525\wd0}}

\def\dashint{\Xint-}

\begin{document}
	\title[Kolmogorov operators]{New perspectives on recent trends for Kolmogorov operators}
	

	\thanks{{\it Aknowledgments:} 
	The authors are members of the Gruppo Nazionale per l’Analisi Matematica, la Probabilità e le loro Applicazioni (GNAMPA) of the Istituto Nazionale di Alta Matematica (INdAM). 
	The first and third authors are partially supported by the INdAM - GNAMPA 
		project ``Variational problems for Kolmogorov equations: long-time analysis and regularity estimates'', 
		CUP\!$\_$\,E55F22000270001.}
	
	\author[F. Anceschi]{Francesca Anceschi}
	\address{Francesca Anceschi\\Dipartimento di Ingegneria Industriale e Scienze Matematiche\\
		Universit\'a Politecnica delle Marche\\ 
		Via Brecce Bianche, 12, 60131 Ancona, Italy}
	\email{\url{f.anceschi@staff.univpm.it}}
	\author[M. Piccinini]{Mirco Piccinini} 
	\address{Mirco Piccinini\\Dipartimento di Scienze Matematiche, Fisiche e Informatiche\\
		Universit\'a degli Studi di Parma\\ 
		Parco Area delle Scienze 53/a, Campus, 43124 Parma, Italy}
	\email{\url{mirco.piccinini@unipr.it}}

	\author[A. Rebucci]{Annalaura Rebucci}
	\address{Annalaura Rebucci\\Dipartimento di Scienze Fisiche, Informatiche e Matematiche \\
		Universit\'a degli Studi di Modena e Reggio Emilia\\
		Via Campi 213/B, 41125 Modena, Italy}
		\email{\url{annalaura.rebucci@unipr.it}}
	 	
		\maketitle
	
	\begin{abstract}
		\noindent
		After carrying out an overview on the non Euclidean geometrical setting suitable 
		for the study of Kolmogorov operators with rough coefficients, we list some properties of the functional space $\W$, 
		mirroring the classical $H^1$ theory for uniformly elliptic operators. 
		Then we provide the reader with the proof of a new Sobolev 
		embedding for functions in $\W$.
		Additionally, after reviewing recent results regarding weak regularity theory, 
		we discuss some of their recent applications to real life problems arising both in Physics and in 
		Economics. Finally, we conclude our analysis stating some recent results regarding 
		the study of nonlinear nonlocal kinetic Kolmogorov-Fokker-Planck operators.  
	\end{abstract}
	
	\setcounter{tocdepth}{1}  	
   \tableofcontents

\section{Introduction}
\label{intro}
	This work is devoted to the study of ultraparabolic second order partial differential equations of Kolmogorov type of the form
	\begin{eqnarray}\label{defL}
		 &&	\L u (x,t):=\sum_{i,j=1}^{m_0}\partial_{x_i}\left(a_{ij}(x,t)\partial_{x_j}u(x,t)\right)+\sum_{i,j=1}^N b_{ij}x_j\partial_{x_i}u(x,t)-\partial_t u(x,t)\notag\\
			&&\qquad \qquad \qquad \qquad \qquad+\sum_{i=1}^{m_0}b_i(x,t)\partial_i u(x,t)+c(x,t)u(x,t)=f(x,t),
	\end{eqnarray}
	where $z=(x,t)=(x_1,\ldots,x_N,t)\in \R^{N+1}$ and $1 \leq m_0 \leq N$. 
	Furthermore, matrices $A_0=(a_{ij}(x,t))_{i,j=1,\ldots,m_0}$ and $B=(b_{ij})_{i,j=1,\ldots,N}$ satisfy the following structural assumptions.

\begin{itemize}
\item[\textbf{(H1)}] The matrix $A_0$ is symmetric with real measurable entries. Moreover, 
	there exist two positive constants $\lambda$ and $\Lambda$ such that
	\begin{eqnarray*}\label{hypcost}
			\lambda |\xi|^2 \leq \sum_{i,j=1}^{m_0}a_{ij}(x,t)\xi_i\xi_j \leq \Lambda|\xi|^2
	\end{eqnarray*}
	for every $(x,t) \in \R^{N+1}$ and $\xi \in \R^{m_0}$. The matrix B has constant entries.
\end{itemize}
	
Despite the degeneracy of $\L$ whenever $m_0 < N$, its first order part is a strongly regularizing operator. Indeed, 
it is known that, under suitable structural assumptions on the matrix $B$, the \textit{principal part operator} $\L_0$ of $\L$
\begin{eqnarray*}\label{defK}
	\L_0 u(x,t):=\sum_{i=1}^{m_0}\partial^2_{x_i} u(x,t) + \sum_{i,j=1}^N b_{ij}x_j\partial_{x_i}u(x,t)-\partial_t u(x,t),\qquad (x,t) \in \R^{N+1},
\end{eqnarray*}
is hypoelliptic (i.~\!e. every distributional solution $u$ to $\L_0 u = f$ defined in 
some open set $\Omega \subset \R^{N+1}$ belongs to $C^\infty(\Omega)$ and it is a classical solution whenever $f \in C^\infty(\Omega)$). 
Hence, in the sequel, we rely on the following assumption.

\medskip

\begin{itemize}
\item[\textbf{(H2)}] The \textit{principal part operator $\L_0$ of $\L$} is hypoelliptic
and homogeneous of degree $2$ with respect to the family of dilations $\left( \d_{r} \right)_{r>0}$ introduced in \eqref{gdil}.
\end{itemize}

This latter assumption is clearly satisfied whenever $\L_0$ is uniformly parabolic, which corresponds to the choice $m_0=N$ and $B\equiv \mathds{O}$. 
Actually, in this case the principal part operator $\L_0$ coincides with the heat operator, which is known to be hypoelliptic. 
Moreover, \cite[Propositions 2.1 and 2.2]{LP} imply that assumption \textbf{(H2)}
is equivalent to assume there exists a basis of $\R^N$ with respect to which $B$ takes the form
\begin{equation}
	\label{B}
	B =
	\begin{pmatrix}
		\OO   &   \OO   & \ldots &    \OO   &   \OO   \\  
		B_1   &    \OO  & \ldots &    \OO    &   \OO  \\
		\OO    &    B_2  & \ldots &  \OO    &   \OO   \\
		\vdots & \vdots & \ddots & \vdots & \vdots \\
		\OO    &  \OO    &    \ldots & B_\k    & \OO
	\end{pmatrix},
\end{equation}
where every $B_j$ is a $m_{j} \times m_{j-1}$ matrix of rank $m_j$, with $j = 1, 2, \ldots, \k$, 
\begin{equation} \label{mj}
	m_0 \ge m_1 \ge \ldots \ge m_\k \ge 1 \hspace{5mm} \text{and} \hspace{5mm} 
	\sum \limits_{j=0}^\k m_j = N.
\end{equation} 
Thus, from now on we assume $B$ has the canonical form \eqref{B}.
Moreover, by introducing the \textit{spatial homogeneous dimension of} $\R^{N+1}$, a quantity defined as
\begin{equation}
	\label{hom-dim}
	Q = m_0 + 3 m_1  + \ldots + (2\k+1) m_\k,
\end{equation}
and the \textit{homogeneous dimension of} $\R^{N+1}$ defined as $Q+2$, we are now in a position to state our last assumption regarding the integrability of $b$, $c$ and of the source term $f$.

\medskip

\begin{itemize}
\item[\textbf{(H3)}] $c, f \in L^q_{loc}(\O)$, with $q > \frac{Q+2}{2}$, and $b \in \left( L^{\infty}_{loc} (\O) \right)^{m_0}$.
\end{itemize}

\medskip

From now on, we denote by $D=(\partial_{x_1},\ldots,\partial_{x_N})$, $\langle\cdot,\cdot \rangle$, and $\div$ 
the gradient, the inner product and the divergence in $\R^N$, respectively. Moreover, 
$D_{m_0}=(\partial_{x_1},\ldots,\partial_{x_{m_0}})$ and $\div_{m_0}$
stand for the partial gradient and the partial divergence in the first $m_0$ components, respectively. 

If $a_{ij}$ are the coefficients appearing in \eqref{defL} for every $i,j=1,\ldots,m_0$ and $a_{ij}\equiv 0$ whenever $i > m_0$, or $j>m_0$, then we introduce
\begin{align*}  
	A(x,t)= &\left(a_{ij}(x,t)\right)_{1\leq i,j \leq N}, \, \, \,	 
	b(x,t):=\left(b_1(x,t),\ldots,b_{m_0}(x,t),0,\ldots,0\right),
	 \\ \nonumber 
	 &\qquad \qquad Y:=\sum_{i,j=1}^N b_{ij}x_j\partial_{x_i}u(x,t)-\partial_t u(x,t).
\end{align*}
Thus, we are in a position to rewrite operators $\L$ and $\L_0$ in a compact form
\begin{align*} 
	\L u=\div(A D u) + Yu + \langle b, D u\rangle+ cu \quad \text{and} \quad \L_0 u = \Delta_{m_0}u + Y u ,
\end{align*}
and to introduce a useful example for the class of ultraparabolic operators of type \eqref{defL}.

\begin{ex} \label{kfp-ex}
	A notable prototype belonging to the class \eqref{defL} is the kinetic Kolmogorov-Fokker-Planck equation
	\begin{align} \label{kfp}
		\nabla_{v} \cdot \left(u (v,x,t)\,v + \nabla_v u (v,x,t)\right) + v \cdot \nabla_{x} u(v,x,t) - \partial_t u (v,x,t) = f(v,x,t),
	\end{align}
	where $(v,x,t) \in \R^{2n+1}$. Equation \eqref{kfp} is obtained from \eqref{defL} 
	by choosing $N=2n$, $\k = 1$, $m_0 = m_1= n$ and $c\equiv n$.
	From the physical point of view, Fokker-Planck equations like the one in \eqref{kfp} provide a continuous description of the dynamics of the 
distribution of Brownian test particles immersed in a fluid in thermodynamical equilibrium. More precisely, the distribution function $u$ of a test particle evolves according to the linear equation in \eqref{kfp}, provided that the test particle is much heavier than the molecules of the fluid. In particular, equation \eqref{kfp} is the backward Kolmogorov equation of the stochastic process
\begin{equation}\label{langevin}
\begin{cases}
    & dV_{t} =  \sqrt{2} dW_{t}-V_tdt, \\
	&d X_{t} =  V_t dt,
\end{cases}
\end{equation}
where $\left(W_{t}\right)_{t \ge 0}$ denotes a $n-$dimensional Wiener process. We refer the reader to {\cite{APsurvey}, and the reference therein, for an exhaustive treatment of Fokker-Planck equations, and their applications.
}
\end{ex}

\subsection{Geometrical setting}
\label{preliminaries}
First of all, we describe the most suitable geometrical setting for the study of $\L$, which is not an Euclidean one. Indeed, 
as it was firstly observed by Lanconelli and Polidoro in \cite{LP}, operator $\L_0$ is invariant with respect to left 
translation in the Lie group $\mathds{K}=(\mathds{R}^{N+1},\circ)$, whose group law is defined as 
\begin{equation}
	\label{grouplaw}
	(x,t) \circ (\xi, \tau) = (\xi + E(\tau) x, t + \tau ), \hspace{5mm} (x,t),
	(\xi, \tau) \in \R^{N+1},
\end{equation}
where the exponential of the group is
\begin{equation}\label{exp}
	E(s) = \exp (-s B), \qquad s \in \R.
\end{equation} 
We observe that $\mathds{K}$ is a non-commutative group with zero element $(0, \ldots ,0,0)$ and inverse
\begin{equation*}
(x,t)^{-1} = (-E(-t)x,-t).
\end{equation*}
Additionally, if for a given $\zeta \in \R^{N+1}$ we denote by $\ell_{\z}$ the left traslation on $\mathds{K}=(\R^{N+1},\circ)$ defined as follows
\begin{equation*}
	\ell_{\z}: \R^{N+1} \rightarrow \R^{N+1}, \quad \ell_{\z} (z) = \z \circ z,
\end{equation*}
then it is possible to show $\L_0$ is left invariant with respect to the Lie product $\circ$, i.e.
\begin{equation*}
	\label{ell}
    \L_0 \circ \ell_{\z} = \ell_{\z} \circ \L_0 \qquad \text{or, equivalently,} 
    \qquad \L_0\left( u( \z \circ z) \right)  = \left( \L_0 u \right) \left( \z \circ z \right),
\end{equation*}
for every $u$ sufficiently smooth.

Now, let us focus on assumption \textbf{(H2)}. Indeed, the hypoellipticity of $\L_0$ is implied by H\"{o}rmander's rank condition, 
which was introduced for the first time in \cite{H} and reads as:
\begin{equation}\label{hormander}
{\rm rank\ Lie}\left(\partial_{x_1},\dots,\partial_{x_{m_0}},Y\right)(x,t) =N+1,\qquad \forall \, (x,t) \in \R^{N+1},
\end{equation}
where ${\rm  Lie}\left(\partial_{x_1},\dots,\partial_{x_{m_0}},Y\right)$ denotes the Lie algebra generated by the first order differential operators $\left(\partial_{x_1},\dots,\partial_{x_{m_0}},Y\right)$ computed at $(x,t)$.
On a side note, it is worth noting that requiring
\begin{eqnarray}\label{Cpositive}
	C(t) > 0, \quad \text{for every $t>0$},
\end{eqnarray}
it is equivalent to assume the hypoellipticity for $\L_0$ as in {\bf (H2)},
see \cite[Proposition A.1]{LP}, 
where $E(\cdot)$ is defined in \eqref{exp} and
the covariance matrix is defined as
\begin{equation*}
	C(t) = \int_0^t \hspace{1mm} E(s) \, A_0 \, E^T(s) \, ds.
\end{equation*}

As far as we are concerned with the second half of assumption \textbf{(H2)},
$\L_0$ is invariant with respect to the family of dilations $(\delta_r)_{r>0}$ if  
\begin{equation}
	\label{Ginv}
      	 \L_0 \left( u \circ \delta_r \right) = r^2 \delta_r \left( \L_0 u \right), \quad \text{for every} \quad r>0,
\end{equation}
and for every function $u$ sufficiently smooth. As pointed out in \cite[Proposition 2.2]{LP}, it is possible to read this dilation invariance property 
in the expression of the matrix $B$ in \eqref{B}. Specifically, $\L_0$ satisfies \eqref{Ginv} if and only if $B$ takes the form \eqref{B}.
In this case, it holds
\begin{equation}
	\label{gdil}
	\d_{r} = \text{diag} \left( r \I_{m_0}, r^3 \I_{m_1}, \ldots, r^{2\k+1} \I_{m_\k}, 
	r^2 \right), \qquad \qquad r > 0.
\end{equation}
Taking this definition into account, we introduce a homogeneous norm of degree $1$ with respect to 
$(\d_{r})_{r>0}$ and a corresponding invariant quasi-distance with respect to the group operation \eqref{grouplaw}.
\begin{defn}
	\label{hom-norm}
	Let $\a_1, \ldots, \a_N$ be positive integers such that
	\begin{equation}\label{alphaj}
		\diag \left( r^{\a_1}, \ldots, r^{\a_N}, r^2 \right) = \d_r.
	\end{equation}
	If $\Vert z \Vert = 0$, then we set $z=0$; if $z \in \R^{N+1} 
	\setminus \{ 0 \}$, then we define $\Vert z \Vert = r$, where $r$ is the 
	unique positive solution to the equation
	\begin{equation*}
		\frac{x_1^2}{r^{2 \a_1}} + \frac{x_2^2}{r^{2 \a_2}} + \ldots 
		+ \frac{x_N^2}{r^{2 \a_N}} + \frac{t^2}{r^4} = 1.
	\end{equation*}
	Accordingly, we define the quasi-distance $d$ by
	\begin{equation}\label{def-dist}
		d(z, w) = \Vert z^{-1} \circ w \Vert , \hspace{5mm} z, w \in 
		\R^{N+1}.
	\end{equation}
\end{defn}

We remark that the Lebesgue measure is invariant with respect to the translation group 
associated to $\L_0$, since $\det E(t) = e^{t \hspace{1mm} \text{\rm trace} \,B} = 1$. 
Moreover, by definition, the semi-norm $\Vert \cdot \Vert$ is homogeneous of degree $1$ with respect to $(\d_r )_{r>0}$.
Indeed, 
	\begin{equation*}
		\Vert \d_r (x,t) \Vert = r  \Vert  (x,t) \Vert
		\qquad \forall r >0 \hspace{2mm} \text{and} \hspace{2mm} (x,t) \in \R^{N+1}.
	\end{equation*}
	Since in $\R^{N+1}$ all the norms which are $1$-homogeneous with respect to 
	$(\d_r)_{r>0}$ are equivalent, the one introduced in Definition 
	\ref{hom-norm} is equivalent to other norms, such as the following one
	\begin{equation*}
		\Vert (x,t) \Vert_1 =|t|^{\frac{1}{2}}+|x|, \quad |x|=\sum_{j=1}^N |x_j|^{\frac{1}{\alpha_j}}
	\end{equation*}
	where the exponents $\alpha_j$, for $j=1,\ldots,N$ were introduced in \eqref{alphaj}. 
	Nevertheless, Definition \ref{hom-norm} is usually preferred since its level sets are smooth surfaces. For
	further information on this matter, we refer to \cite{Manfredini}.	
\begin{rem}
In the case of the kinetic Kolmogorov-Fokker-Planck equation \eqref{kfp} in the absence of friction, i.e.
\begin{align} \label{kfp-nofriction}
		\K_0 u(v,x,t):=\Delta_{v} u (v,x,t)  + v \cdot \nabla_{x} u(v,x,t) - \partial_t u (v,x,t) = f(v,x,t), 
	\end{align}
where $(v,x,t) \in \R^{2n+1}$, the Lie group has a quite natural intepretation. Indeed the composition law \eqref{grouplaw} agrees with the \emph{Galilean} change of variables 
\begin{equation}\label{galilean}
    (v,x,t) \circ (v_{0}, x_{0}, t_{0}) = (v_{0} + v, x_{0} + x + t v_{0}, t_{0} + t). 
\end{equation}
It is easy to see that $\K_0$ is invariant with respect to the above change of variables. Specifically, if $w(v,x,t) = u (v_{0} + v, x_{0} + x + t v_{0}, t_{0} + t)$ and $g(v,x,t) = f(v_{0} + v, x_{0} + x + t v_{0}, t_{0} + t)$, then 
\begin{equation}\label{inv-K}
	\K_0 u = f \quad \iff \quad \K_0 w = g \quad \text{for  every} \quad (v_{0}, x_{0}, t_{0}) \in \R^{2n+1}.
\end{equation}
Moreover, $\K_0$ is invariant with respect to the dilation \eqref{gdil}, which in this case takes the simpler form $\delta_r(v,x,t) := (r v, r^3 x, r^2 t)$. Let us remark that the dilation acts as the usual parabolic scaling with respect to variables $v$ and $t$. Moreover, the term $r^3$ in front of $x$ is due to the fact that the velocity $v$ is the derivative of the position $x$ with respect to time $t$.
\end{rem}	

Hence, considering the group law \ap$\circ$'' and the family of dilations $\left(\delta_r\right)_{r>0}$,
we are now in a position to introduce a suitable family of cylinders, starting from the unit past cylinders
\begin{eqnarray*}
	\Q_1 := B_1 \times B_1 \times \ldots \times B_1 \times (-1,0),
	\qquad \widetilde \Q_1 := B_1 \times B_1 \times \ldots \times B_1 \times (-1,0]
\end{eqnarray*}
defined through open balls 
\begin{eqnarray*}
	B_1 = \lbrace x^{(j)}\!\in\R^{m_j} : \vert x \vert \leq 1 \rbrace,
\end{eqnarray*}
where $j=0,\ldots,\kappa$ and $\vert \cdot \vert$ denotes the Euclidean norm in $\R^{m_j}$.
Then, for every $z_0 \in \mathds{R}^{N+1}$ and $r>0$, we set
\begin{eqnarray*}\label{rcylind}
	\Q_r(z_0):=z_0\circ\left(\delta_r\left(\Q_1\right)\right)=\lbrace z \in \mathds{R}^{N+1} \,:\, z=z_0\circ \delta_r(\zeta), \zeta \in \Q_1 \rbrace.
\end{eqnarray*}
We point out that this definition of slanted cylinders admits an equivalent ball representation, see 
\cite[equation (21)]{WZ3}, that it is sometimes preferred. More specifically, there exists a positive constant 
$\overline c$  such that
	\begin{align*}
		 B_{r_{1}}(x^{(0)}_{0}) &\times B_{r_{1}^{3}}(x^{(1)}_{0}) 
		\times \ldots \times
		B_{r_{1}^{2\k+1}}(x^{(\k)}_{0}) \times (t_0-r_{1}^{2}, t_0] \\
		&\subset
		\Q_{r}(z_{0}) 
		\subset B_{r_{2}}(x^{(0)}_{0}) \times B_{r_{2}^{3}}(x^{(1)}_{0}) 
		\times \ldots \times
		B_{r_{2}^{2\k+1}}(x^{(\k)}_{0}) \times (t_0-r_{2}^{2}, t_0],
	\end{align*}
	where $r_{1}= r/\overline c$ and $r_{2} = \overline c r$.
Moreover, since $\det \d_r = r^{Q+2}$, it is true that
\begin{equation*}
	\meas \left( \Q_r(z_0) \right) = r^{Q+2} \meas \left( \Q_1(z_0) \right), \qquad \forall
	\ r > 0, z_0 \in \R^{N+1},
\end{equation*}
where $Q$ is the homogeneous dimension defined in \eqref{hom-dim}. Finally, 
by \cite[Lemma 6]{CPP} there exists a positive constant $\widetilde{c} \in (0,1)$ such that
	\begin{equation}
		\label{cylindereq}
		z \circ \Q_{\widetilde{c}  (r-\r)} \subseteq \Q_r, \qquad \text{for every } 0 < \r < r \le 1
		\quad \text{and} \quad z \in \Q_\r.
	\end{equation}
We refer to \cite{APsurvey, BLU} and the references therein for more informations on this subject.

Finally, we conclude this subsection by recalling the useful notion of H\"older continuous function
in this non Euclidean setting.
\begin{defn}
    \label{holdercontinuous}
    Let $\a$ be a positive constant, $\a \le 1$, and let $\O$ be an open subset of $\R^{N+1}$. We 
    say that a function $f : \O \longrightarrow \R$ is H\"older continuous with exponent $\a$ in $\O$
    with respect to the group $\mathds{K}=(\mathds{R}^{N+1},\circ)$, defined in \eqref{grouplaw}, (in short: H\"older 
    continuous with exponent $\a$, $f \in C^\a_{K} (\O)$) if there exists a positive constant $C>0$ such that 
    \begin{equation*}
        | f(z) - f(\z) | \le C \; d(z,\zeta)^{\a} \qquad { \rm for \, every \, } z, \z \in \O,
    \end{equation*}
where $d$ is the distance defined in \eqref{def-dist}.

To every bounded function $f \in C^\a_{K} (\O)$ we associate the semi-norm
    	\begin{equation*}
       		[ f ]_{C^{\a} (\O)} =  
       		\hspace{1mm} 
       		 \sup \limits_{z, \z \in \O \atop  z \ne \z} \frac{|f(z) - f(\z)|}{d(z,\zeta)^{\a}}.
        \end{equation*}
    Moreover, we say a function $f$ is locally H\"older continuous, and we write $f \in C^{\a}_{K,\loc}(\O)$,
    if $f \in C^{\a}_{K}(\O')$ for every compact subset $\O'$ of $\O$.
\end{defn}

\subsection{Fundamental solution} \label{fun-sol-s}
Another useful tool in the study of \eqref{defL} with rough coefficients is the fundamental solution of its principal part operator $\L_0$. 
Indeed, under assumption \textbf{(H2)}, H\"{o}rmander explicitly constructed in \cite{H} the fundamental solution of $\L_0$, 
that is
\begin{equation*} 
	\Gamma (z,\zeta) = \Gamma (\zeta^{-1} \circ z, 0 ), \hspace{4mm} 
	\forall z, \zeta \in \R^{N+1}, \hspace{1mm} z \ne \zeta,
\end{equation*}
where
\begin{equation*} 
	\Gamma ( (x,t), (0,0)) = \begin{cases}
		\frac{(4 \pi)^{-\frac{N}{2}}}{\sqrt{\text{det} C(t)}} \exp \left( - 
		\frac{1}{4} \langle C^{-1} (t) x, x \rangle - t \, \trace (B) \right), \hspace{3mm} & 
		\text{if} \hspace{1mm} t > 0, \\
		0, & \text{if} \hspace{1mm} t < 0.
	\end{cases}
\end{equation*}
Since assumption \textbf{(H2)} implies that condition \eqref{Cpositive} holds true, the function $\Gamma$ is well-defined. Moreover, since $\L_0$ is dilation invariant with respect to $(\d_r )_{r>0}$, also $\G$ is a homogeneous function of degree $- Q$, namely
\begin{equation*}
	\Gamma \left( \d_{r}(z), 0 \right) = r^{-Q} \hspace{1mm} \Gamma
	\left( z, 0 \right), \hspace{5mm} \forall z \in \R^{N+1} \setminus
	\{ 0 \}, \hspace{1mm} r > 0.
\end{equation*}
This property implies a $L^p$ estimate for Newtonian potentials (see for instance \cite{AMP}).
Hence, by defining the $\G-$\textit{potential} of a function $f \in L^1(\R^{N+1})$ as
\begin{equation}
	\label{L0p}
	\G (f) (z) = \int_{\R^{N+1}} \G (z, \z ) f(\z) \dd\z, \qquad 
	z \in \R^{N+1},
\end{equation}
we are able to introduce a function $\G (D_{m_0} f): \R^{N+1} \longrightarrow \R^{m_0}$, that is 
well-defined for any $f \in L^p (\R^{N+1})$, at least in the distributional sense, see \cite{CPP}. Indeed
\begin{equation} 
	\label{pot0}
	\G (D_{m_0} f ) (z) := - \int_{\R^{N+1}} D^{(\z)}_{m_0} \G (z, \z) \,
	f(\z) \, \dd\z,
\end{equation}
where $D^{(\z)}_{m_0} \G (z, \z)$ is the gradient with respect to $\x_1, 
\ldots, \x_{m_0}$. Thus, by directly applying \cite[Proposition 3]{CPP}, with $\a = 1$ and $\a=2$ when considering the $\Gamma$-potential for $f$ and $D_{0}f$, respectively, 
it is possible to derive explicit potential estimates for \eqref{L0p} and \eqref{pot0}.  
\begin{corol}
	\label{corollary}
	Let $f \in L^p (\Q_r)$. There exists a positive constant $c = c(T,B)$ 
	such that
	\begin{align*} 
		\Vert \G (f) \Vert_{L^{p**}(\Q_r)} &\le c \Vert f 
		\Vert_{L^{p}(\Q_r)}, \qquad \quad
		\Vert \G (D_{m_0} f) \Vert_{L^{p*}(\Q_r)} \le c 
		\Vert f \Vert_{L^{p}(\Q_r)},
	\end{align*}
	where $\frac{1}{p*}= \frac{1}{p} - \frac{1}{Q+2}$ and $\frac{1}{p**}= 
	\frac{1}{p} - \frac{2}{Q+2}$.
\end{corol}

\subsection{Plan of the paper}
This work is organized as follows. Section \ref{preliminaries-f} is devoted to the study of the functional space $\W$.
In particular, after recalling all known results regarding the space $\W$, we prove a Sobolev embedding for functions belonging to it.
In Section \ref{weak-reg} we provide an overview on the 
De Giorgi-Nash-Moser weak regularity theory in this framework. 
Section \ref{applications} is devoted to the application of these results to the study of real life problems  	
arising both in Physics and Economics.
Finally, we conclude with Section \ref{nnkfp}, where we discuss recent trends of nonlinear nonlocal Kolmogorov type operators. 

\section{Functional setting}
\label{preliminaries-f}
From now on, we consider a set $\O = \O_{m_{0}} \times \O_{N - m_{0} + 1}$ of $\R^{N+1}$, where
$\O_{m_{0}}$ is a bounded {\color{blue}$C^1$ domain of $\R^{m_{0}}$} and 
$\O_{N-m_{0} + 1}$ is a bounded domain of $\R^{N - m_{0} + 1}$. 
Then, according to the scaling introduced in \eqref{gdil}, we split the coordinate $x\in\R^N$ as
\begin{equation}\label{split.coord.RN}
	x=\big(x^{(0)},x^{(1)},\ldots, x^{(\kappa)}\big), \qquad x^{(0)}\!\in\R^{m_0}, \quad x^{(j)}\!\in\R^{m_j}, \quad j\in\{1,\ldots,\kappa\},
\end{equation}
where every $m_{j}$ is a positive integer satisfying conditions exposed in \eqref{mj}.

Furthermore, we denote by $\D(\O)$ the set of $C^\infty$ functions compactly supported in $\O$ 
and by $\D'(\O)$ the set of distributions in $\O$. From now on, $H^{1}_{x^{(0)}}$ is the Sobolev space of functions $u \in  L^{2} (\O_{m_{0}})$ with 
distribution gradient $D_{m_0}u$ lying in $( L^{2} (\O_{m_{0}}) )^{m_{0}}$, i.e. 
\begin{equation} \label{h1-def}
	H^{1}_{x^{(0)}} (\Om) := \left\{ u \in L^{2} (\O_{m_{0}}) : \, D_{m_{0}} u \in ( L^{2} (\O_{m_{0}}) )^{m_{0}}
	\right\},
\end{equation}
paired with the norm
\begin{equation} \label{h1-norma}
	\| u \|^2_{H^{1}_{x^{\small (0)}} (\Om) } := \| u \|^2_{L^{2} (\O_{m_{0}}) } + \| D_{m_{0} } u \|^2_{L^{2} (\O_{m_{0}})}.
\end{equation}
Now, let $H^{1}_{c,x^{(0)}}$ denote the closure of $C^\infty_{c}(\O_{m_{0}})$ in the norm of $H^{1}_{x^{(0)}}$, and recall that {\color{blue}$C^\infty_{c}(\overline \O_{m_{0}})$ is dense in $H^{1}_{x^{(0)}}$ since $\O_{m_{0}}$ is a bounded $C^1$ domain by assumption.} Moreover, $H^{1}_{c,x^{(0)}}$ is a reflexive Hilbert space and thus we may consider its dual space 
$$
\left( H^{1}_{c,x^{(0)}} \right)^{*} = H^{-1}_{x^{(0)}} \quad \text{and} \quad 
\left( H^{-1}_{x^{(0)}} \right)^{*} = H^{1}_{c,x^{(0)}},
$$
where the adopted notation is the classical one. Hence, we denote by $H^{-1}_{x^{(0)}}$ the dual of 
$H^{1}_{c,x^{(0)}}$ acting on functions in $H^{1}_{c,x^{(0)}}$ through the duality pairing 
$\langle \cdot | \cdot \rangle := \langle \cdot | \cdot \rangle_{H^{1}_{x^{(0)}},H^{1}_{c,x^{(0)}}} $.
From now on, we consider the shorthand notation $L^{2}H^{-1}$ to denote $L^{2}\left( \O_{N-m_{0}+1}; H^{-1}_{c,x^{(0)}} \right)$.

Then in a standard manner, see \cite{AM, AR-harnack, LN}, we introduce the space of functions $\W$ as the closure of 
$C^{\infty} (\overline \O)$ in the norm 
\begin{equation}
	\label{normW}
	\| u \|^2_{\W} = \| u \|^2_{L^2 \left( \O_{N-m_{0} + 1};H^1_{x^{(0)}} \right)} + \| Y u \|^2_{L^2\left( \O_{N-m_{0} + 1};H^{-1}_{x^{(0)}}\right)},
\end{equation}
which can explicitly be computed as
\begin{align*}
	\| u \|^2_{\W} =  \int_{\O_{N-m_{0} + 1}} \| u(\cdot, y,t) \|_{H^1_{x^{(0)}}}^{2} \dd y \, \dd t  + 
				  \int_{\O_{N-m_{0} + 1}}  \| Y u (\cdot, y,t) \|^2_{H^{-1}_{x^{(0)}}} \dd y \, \dd t ,
\end{align*}
where $y = (x^{(1)},\ldots, x^{(\kappa)})$. In particular, it is possible to infer that $\W$ is a Banach space and we recall that the dual of
$L^{2}(\O_{N-m_{0}+1}; H^{1}_{c,x^{(0)}})$ satisfies 
\begin{align*}
	&\left( L^{2}(\O_{N-m_{0}+1}; H^{1}_{c,x^{(0)}}) \right)^{*} =  L^{2}(\O_{N-m_{0}+1}; H^{-1}_{c,x^{(0)}}) 
	\quad \text{and} \quad  \\
	&\qquad \qquad  \qquad \qquad \left( L^{2}(\O_{N-m_{0}+1}; H^{-1}_{c,x^{(0)}})   \right)^{*} = L^{2}(\O_{N-m_{0}+1}; H^{1}_{c,x^{(0)}}) .
\end{align*}

This functional setting was firstly proposed in \cite{AM} for the study of weak regularity theory for the Kolmogorov-Fokker-Planck equation, 
see Example \ref{kfp-ex}. Later on, it was considered 
in \cite{LN} for the study of well-posedness results for a Dirichlet problem in the kinetic setting. 
Finally, two of the authors extended this framework in \cite{AR-harnack} to the ultraparabolic setting considered in this work.

Notice that the major issue one has to tackle with when dealing with $\W$ is the duality pairing between $L^{2}H^{1}$ and $L^{2}H^{-1}$. Usually, when considering bounded domains, this problem is overcome  observing that for every open subset $A \subset \R^{n}$ and for every 
function $g \in H^{-1}(A )$, there exist two functions
$H_0$, $H_1 \in L^2(A)$ such that 
\begin{equation*}
	g = \div_{m_{0}} H_{1} + H_{0} \qquad \text{and} \qquad 
	\Vert H_0 \Vert_{L^2(A)}+\Vert H_1 \Vert_{L^2(A)} \leq 
	2\Vert g \Vert_{H^{-1}(A)},
\end{equation*}
see, for example,~\cite[Chapter 4]{SK}.

Moreover, in recent years the community focused on providing suitable functional inequalities for
functions belonging to $\W$. In particular, there was necessity to prove a suitable Poincaré and 
a suitable Sobolev embedding in this framework. As far as we are concerned with the first one, 
the proof of a weak Poincaré inequality was recently achieved in \cite{AR-harnack, GI}. Therefore, we will only recall 
its statement and a scheme of its proof. On the other hand, the Sobolev embedding is still not available in literature
and for this reason we provide the reader with its full proof.

\subsection{Poincaré inequality} 
An useful tool we have at our disposal when considering functions belonging to $\W$ is a weak Poincaré inequality 
proved for the first time in \cite{GI} for the kinetic case and later on extended in \cite{AR-harnack} to the setting considered in this work. 
The idea is to firstly derive a local Poincaré inequality in terms of an error function $h$ defined as the solution of a suitable Cauchy problem 
\begin{equation*}
	\left\{ 
	\begin{array}{ll} 
		\tilde{\K} h=u \tilde{\K} \psi,\quad &\textit{in $  \mathds{R}^{N} \times (-\r^2,0)$}\\
		h=0,\quad &\textit{in $  \mathds{R}^{N}\times \lbrace -\r^2 \rbrace$}
	\end{array} \right.
\end{equation*}
where $\rho > 0$, $\psi$ is a given cut-off function and $\tilde{\K} $ is an auxiliary operator defined as
\begin{equation*}
	\tilde{\K} u(x,t):=-\sum_{i=1}^{m_0}\partial^2_{x_i} u(x,t) - \sum_{i,j=1}^N b_{ij}x_j\partial_{x_i}u(x,t)+\partial_t u(x,t),
	\qquad (x,t) \in \R^{N+1}. 
\end{equation*}
We observe that the auxiliary operator $\tilde{\K}$ is chosen in accordance with the definition of $\W$, where only the partial gradient $D_{m_0}$ and the Lie derivative $Y$ appear.
On one hand, completing the proof by explicitly controlling the error $h$ through the $L^{\infty}$ norm of the function $u$ allows us to avoid studying functional properties of 
weak solutions to \eqref{defL} and to obtain a purely functional result. On the other hand, it is immediately clear that our inequality only holds for \textit{bounded} functions belonging to $\W$.  

In order to state this result, we first need to introduce the following sets
\begin{align}\label{Qzero}
	\Q_{zero}&=\lbrace (x,t) :  |x_j|\leq \eta^{\alpha_j}, j=1,\ldots,N, -1-\eta^2 < t \leq -1\rbrace,\\ \nonumber
	\Q_{ext}&=\lbrace (x,t) :  |x_j|\leq 2^{\alpha_j}R, j=1,\ldots,N, -1-\eta^2 < t \leq 0\rbrace,
\end{align}
where $R>1$, $\eta \in (0,1)$, exponents $\alpha_j$, for $j=1,\ldots,N$, are defined in \eqref{alphaj}; $\Q_{zero}$ and $\Q_{ext}$ are introduced via the ball representation. 
\begin{theorem}[Weak Poincaré inequality]\label{weak-poincare}
	Let $\eta \in (0,1)$; let $\Q_{zero}$ and $\Q_{ext}$ be defined as in \eqref{Qzero}. Then there exist $R>1$ and $\vartheta_0 \in (0,1)$ such that for any non-negative function $u \in \W$ such that $u \leq M$ in $\Q_1= B_1 \times B_1 \times \ldots \times B_1 \times (-1,0)$ for a positive constant $M$ and 
\begin{equation*}
	\left\vert \lbrace u= 0 \rbrace \cap \Q_{zero}\right\vert \geq \frac{1}{4} \left\vert \Q_{zero}\right\vert,
\end{equation*}	
we have
	\begin{equation*}
		\Vert (u-\vartheta_0 M)_+ \Vert_{L^2(\Q_1)}\leq C_P\left(\Vert D_{m_0} u \Vert_{L^2(\Q_{ext})}+\Vert Y u \Vert_{L^2H^{-1}(\Q_{ext})} \right),
	\end{equation*} 
	where $C>0$ is a constant only depending on $Q$.
\end{theorem}
\noindent
The notation we consider here needs to be intended in the sense of \eqref{normW}. 
In particular, we have that $L^2H^{-1}(\Q_{ext})$ is short for 
\begin{equation*}
	L^2(B_{2^3 R}\times \ldots\times B_{2^{2\kappa+1} R}\times(-1-\eta^2,0],H^{-1}_{x^{(0)}}(B_{2 R})),
\end{equation*}
 where we split $x=\big(x^{(0)},x^{(1)},\ldots, x^{(\kappa)}\big)$ according to \eqref{split.coord.RN}.

\subsection{Sobolev inequality}

Here, we give proof to a Sobolev inequality for functions belonging to $\W$. We refer the interested reader for further Sobolev-type embeddings for kinetic Kolmogorov equations to~\cite{PP22} and to~\cite{Pes22} for interpolations results in the non Euclidean geometrical setting of Kolmogorov equations.

Now, in order to prove our desired Sobolev inequality we firstly recall its classical formulation for functions belonging to 
the space $H^1_{x^{\small (0)}}(\Om)$, that we earlier introduced in
\eqref{h1-def} alongside with its norm \eqref{h1-norma}. 
In order to do this, we recall that when $2 < m_0$ we may introduce the Sobolev exponent
\begin{equation*}
	2^* = \frac{2m_0}{m_0 - 2}, \quad \text{such that } \, \frac{1}{2^*} = \frac12 - \frac{1}{m_0}, \, \, 2^* > 2. 
\end{equation*}

	\begin{theorem}[Corollary 9.14 of \cite{Brezis}] \label{sob}
		Let $\Om$ be a bounded open subset of $\mathds{R}^{m_0}$
		with $C^{1}$ boundary and $m_0 > 2$.
		If $u \in H^1_{x^{\small (0)}}(\Om)$, then $u \in L^{q}(\Om)$
		with $q \in [2, 2^*]$, and the following estimate holds
		\begin{align*}
			\| u \|_{L^{q}(\Om)} \le C \| D_{m_0}u \|_{L^2(\Om)},
		\end{align*}
		where $C$ is a constant only depending on $m_0$ and $\Oo$. 
	\end{theorem}
	
Then, by following the approach proposed in \cite{AP-boundedness}, 
we prove a new Sobolev embedding for functions belonging to $\W$. 
It is our belief that, as in the nonlinear nonlocal kinetic framework discussed in \cite{AP-boundedness} and later
on subject of Section \ref{nnkfp}, the following embedding may lead to an improvement in
the study of the De Giorgi-Nash-Moser weak regularity theory, see Section \ref{weak-reg}.
\begin{theorem}
\label{sob-nostra}
	Let~$\O= \O_{N-m_0+1} \times \O_{m_0}$
	be a bounded open subset of $\R^{N+1}$,
	where $\O_{m_{0}}$ is a bounded $C^1$ domain of $\R^{m_{0}}$ and 
	$\O_{N-m_{0} + 1}$ is a bounded Lipschitz domain of $\R^{N - m_{0} + 1}$.
	 Let $u \in \W$ and $m_0 > 2$. Then for every $q \in [2,2^*]$ the following inequality holds
	 \begin{align*}
		&\int_{\O_{m_0}} \left| \,  \dashint_{\O_{N-m_0+1}}  u (x^{(0)},y,t) \, \dd y \, \dd t \right|^q \dd x^{(0)} \\*
		&\qquad
		\le  C  \left( \int_{\O_{m_0}} \dashint_{\O_{N-m_0+1}}  | D_{m_0} u (x^{(0)}, y,t)|^2 
			\, \dd y \, \dd t \, \dd x^{(0)} \right)^{\frac{q}{2}}.
	\end{align*}
\end{theorem}
\begin{proof}
	Let~$u \in L^2(\O_{N-m_0+1};H^1_{x^{(0)}} (\O_{m_0}))$. We define the mean of~$u$ 
	with respect to the variables~$y$ and~$t$ as
	\begin{equation*}
		(u)_{y,t}(x^{(0)}):= \dashint_{\Onm} u(x^{(0)},y,t) \, \dd y \, \dd t.
	\end{equation*}
	Then,~$(u)_{y,t} \in H^1_{x^{(0)}} (\O_{m_0})$. Indeed, by integral's monotonicity property and Jensen's Inequality, we have
	\begin{align*}
		&\int_{\Om}  |(u)_{y,t} (x^{(0)}) |^2 \dd x^{(0)} \le 
		\int_{\O_{m_0}} \dashint_{\O_{N-m_0+1}}  |u(x^{(0)}, y,t) |^2 \, \dd y \, \dd t \, \dd x^{(0)} , \quad \text{and} \\
		&\int_{\Om}  | D_{m_0} (u)_{y,t} (x^{(0)}) |^2 \dd x^{(0)} \le
		\int_{\Om}  \dashint_{\O_{N-m_0+1}} | D_{m_0} u (x^{(0)},y,t) |^2 \, \dd y \, \dd t \, \dd x^{(0)} .
	\end{align*}
	Now, let $q \in [2,2^*]$. Then, by applying Theorem~\ref{sob} to~$(u)_{y,t}$ we get
	\begin{align*}
		\int_{\O_{m_0}} |(u)_{y,t} (x^{(0)})|^q \dd x^{(0)} 
		&\le C  \left( \int_{\O_{m_0}} | D_{m_0} (u)_{y,t} (x^{(0)}) |^2 \dd x^{(0)} \right)^{\frac{q}{2}} \\
		&\le  C  \left( \int_{\O_{m_0}} \dashint_{\O_{N-m_0+1}}  | D_{m_0} u (x^{(0)}, y,t)|^2 
			\, \dd y \, \dd t \, \dd x^{(0)} \right)^{\frac{q}{2}}.
	\end{align*}
\end{proof}

	\begin{rem}
	We observe that in the case~$m_{0} \le 2$ an analog result of Theorem~\ref{sob-nostra} can be proved. Indeed, when~$m_{0} =1$ the function~$(u)_{y,t} \in H^1_{x^{(0)}}(\Om)$ is absolutely continuous, whereas the case when~$m_{0}=2$ can be treated via the Rellich-Kondrachov embedding since~$(u)_{y,t} \in H^1_{x^{(0)}}(\Om) \subset L^q(\Om)$ for any~$q \in [2,\infty)$.
    \end{rem}

\section{De Giorgi-Nash-Moser weak regularity theory}
\label{weak-reg} 
The extension of the De Giorgi-Nash-Moser weak regularity theory to the class of ultraparabolic equations in divergence form of type \eqref{defL} had been an open problem for decades. 
A first breakthrough in this direction was obtained by Pascucci and Polidoro in \cite{PP}, where the authors proved the Moser's iterative scheme for 
\textit{strong weak solutions}, i.e. weak solutions $u \in L^2$ such that $D_{m_0} u, Y u \in L^2$. Then, later on, this result was extended to the nondilation invariant case 
in \cite{APR,CPP}. 

Based on these local boundedness results, the local H\"older continuity for strong weak solutions was later on addressed by Wang and Zhang in \cite{WZ4, WZ3} for the specific case of 
Kolmogorov-Fokker-Planck equation \eqref{kfp}. Subsequently, the procedure was extended to the ultraparabolic setting in the unpublished paper \cite{WZ-preprint}. 
The method considered in these works is based on the combination of Sobolev and Poincaré inequalities constructed for 
strong weak solutions, alongside with qualitative properties of a suitably chosen $G$ function. Then the local H\"older continuity is recovered by providing an estimate of the oscillations
following Kruzkhov's level set method \cite{Kruzhkov}.

In recent years, the interest of the community began to focus on the extension of these regularity results to weak solutions 
belonging to the space $\W$, that are defined as follows.
\begin{defn}\label{weak-sol2}
A function $u \in \W$ is a weak solution to \eqref{defL} with source term $f \in L^2(\O)$ if for every non-negative test function $\phi \in \D(\O)$, we have
\begin{align}\label{kolmo}
	   \int_{\O} - \langle A Du, \dd\phi \rangle - uY\phi   + \langle b , Du \rangle \phi + c u \phi= \int_{\O} f \phi.
	\end{align}
In the sequel, we will also consider weak sub-solutions to \eqref{defL}, namely functions $u \in \W$ that satisfy the following inequality
\begin{align}\label{kolmo-sub}
	   \int_{\O} - \langle A Du, \dd\phi \rangle - uY\phi + \langle b , Du \rangle \phi + c u \phi  \overset{(\leq)}{\geq} \int_{\O} f \phi,
	\end{align}
	for every non-negative test function $\phi \in \D(\O)$. A function $u$ is a super-solution to \eqref{defL} if
	it satisfies \eqref{kolmo-sub} with $(\leq)$.
\end{defn}

A first attempt in this direction is represented by the seminal paper \cite{GIMV}, where the authors propose a non 
constructive proof of a Harnack inequality for weak solutions to the kinetic Kolmogorov-Fokker-Planck equation \eqref{kfp}. 
This approach
is based on classical energy estimates and apriori fractional estimates proved in \cite{Bouchut}. Driven by the aim of simplifying and extending the proof proposed in \cite{WZ-preprint}, various authors recently 
suggested alternative proofs both for the Harnack inequality and the H\"older continuity in the kinetic setting, 
see \cite{GI, GM, IS-weak, Silvestre1, Silvestre2, Zhu}. 

It was only recently that the weak regularity theory was extended to the ultraparabolic case in \cite{AR-harnack}
by two of the authors, and their main results can be stated after the introduction of these two sets:
\begin{align*}
	\Q_+= \delta_\omega \left( \widetilde \Q_1\right) &= B_{\omega} \times B_{\omega^3} \times 
	\ldots \times B_{\omega^{2\k+1}} \times (- \omega^2 , 0] \quad \text{and} \\
	\widetilde \Q_- = (0, \ldots, 0, - 1 + 2 \rho^2) \, \circ \, &\delta_\rho \left( \Q_1\right) 
	=  B_{\rho} \times B_{\rho^3} \times 
	\ldots \times B_{\rho^{2\k+1}} \times (-1 + \rho^2, -1 +2 \rho^2).
\end{align*}
\begin{theorem}[Harnack inequality]
	\label{harnack-thm}
	Let $u$ be a non-negative weak solution to $\L u = f$ in 
	$\O\supset \widetilde \Q_1$ under the assumptions \textbf{(H1)}-\textbf{(H2)}-\textbf{(H3)}. 
	Then	 we have
	\begin{eqnarray*}
		\sup \limits_{\widetilde \Q_{-}} u \, \le \, C \left( \inf_{\Q_+} u + \Vert f \Vert_{L^q(\Q^{0})} \right),
	\end{eqnarray*}
	where $0 < \omega < 1$ is given by Theorem \ref{weak-harnack} and
	$0 < \rho < \frac{\omega}{\sqrt2}$. 
	Finally, the constants $C$, $\omega$, $\rho$ only depend on the homogeneous dimension $Q$ defined in 
	\eqref{hom-dim}, $q$ and on the ellipticity constants $\lambda$ and $\Lambda$ 
\end{theorem}

\begin{theorem}[H\"older regularity]
	\label{local-holder}
	There exists $\a \in (0,1)$ only depending on dimension $Q$, $\l$, $\Lambda$ such that all weak solutions 
	$u$ to \eqref{defL} under assumption \textbf{(H1)}-\textbf{(H2)}-\textbf{(H3)} in $\O \supset \Q_{1}$ satisfy
	\begin{equation*}
		[ u ]_{C^{\a} (Q_{\frac12})} \, \le C \left( \| u \|_{L^{2}(\Q_{1})} + \| f \|_{L^{q}(\Q_{1})} \right),
	\end{equation*}
	where the constant $C$ only depends on the homogeneous dimension $Q$ defined in \eqref{hom-dim}, $q$
	and the ellipticity constants $\l$ and $\Lambda$.
\end{theorem}
Note that the estimates above can be stated and scaled 
in any arbitrary cylinder $\Q_r(z_0)$. 
Furthermore, these results are comparable with the ones obtained in \cite{WZ-preprint}, 
but framework and methodology are different.
Indeed, in \cite{AR-harnack} the authors apply the technique proposed by Guerand and Imbert in \cite{GI}, already 
previously considered by Imbert and Silvestre for the Boltzmann equation in \cite{IS-weak}, 
based on the combination of a weak Poincaré inequality (Theorem \ref{weak-poincare}) 
for functions belonging to the space $\W$, with a $L^2-L^\infty$ estimate for weak sub-solutions (Theorem \ref{boundedness}) and a weak 
Harnack inequality for weak super-solutions (Theorem \ref{weak-harnack}). 
For the sake of completeness, these two tools will be briefly discussed in the following.

\subsection{Local boundedness estimates}
\label{moser}
The proof of this result is obtained via the extension of the Moser iterative scheme 
introduced in \cite{M4} for the parabolic setting and it is based on the iterative combination of a Caccioppoli and a Sobolev inequality. 
When dealing with the ultraprabolic setting of our interest, the degeneracy of the diffusion part only allows us to estimate the partial gradient $D_{m_{0}} u$ of the solution through a Caccioppoli type inequality, also known as energy estimate. Moreover, in accordance with our definition of weak solution, $u$ does not lie in a classical Sobolev space. 
Nevertheless, as firstly observed in \cite{PP}, it is true that
\begin{equation*}
	\L_0 u = \left( \L_0 - \L \right) u + f = \div_{m_{0}} \left( \left( \I_{m_{0}} - A \right) D_{m_{0}}u \right)
	+f .
\end{equation*}
Hence, as pointed out at \cite[p. 396]{PP}, it seems quite natural to consider a representation formula for solutions to \eqref{defL} in terms of the fundamental solution of the principal part operator $\L_0$ to prove a Sobolev embedding for solutions to \eqref{defL}. This is very convenient because we have at our disposal an explicit expression of the fundamental solution of $\L_0$, alongside with potential estimates for it,
see Subsection \ref{fun-sol-s}.

In literature, we find various extensions of the Moser's iterative scheme to Kolmogorov operators of the type $\L$, see for instance \cite{CPP, PP, WZ-preprint}. The most recent one is proved in \cite{AR-harnack} for the functional setting $\W$ and it reads as follows. 
\begin{theorem}\label{boundedness}
	 Let $z_0 \in \O$ and $0 < \frac{r}{2} \le \r 
	< r \leq 1$, be such that $\overline{\Q_r(z_0)}\subseteq \O$. 
	Let $u$ be a non-negative weak solution to $\L u = f$ in $\O$ under 
	assumptions \textbf{(H1)}-\textbf{(H2)}-\textbf{(H3)}.
	Then for every $p \ge 1$ there exists two positive constants $C = C \left( p,\l,\Lambda,Q,
		\parallel b \parallel_{L^q(\Q_r(z_0))},\parallel c \parallel_{L^q(\Q_r(z_0))} \right)$, such that 
	\begin{equation*}
		\sup_{\Q_{\r}(z_0)} u_l^p \, \le \, \frac{ C }{ (r -\r)^{\frac{Q+2}{\beta}} } \| u_l^p \|_{L^{\beta} (\Q_{r}(z_0)) },
	\end{equation*}
	where $\beta = \frac{q}{q-1}$, $q$ introduced in \textbf{(H3)} and $u_l := u + \| f \|_{L^q(\Q_r)}$.
	 The same statement holds true if $u$ is a non-negative weak sub-solution to \eqref{defL} for $p \ge 1$;
	if $u$ is a non-negative weak super-solution to \eqref{defL} for $0 < p < \frac12$. In particular, by choosing $p=1$, for every sub-solution to 
\eqref{defL} it holds
	\begin{equation*}
		\sup_{\Q_{\r}(z_0)} u \, \le \, \frac{ C }{ (r -\r)^{\frac{Q+2}{\beta}} } \left( \| u \|_{L^{\beta} (\Q_{r}(z_0)) } + 
		 \| f \|_{L^q(\Q_r)} \right) ,
	\end{equation*}
\end{theorem}
Firstly, we observe the above statement holds true also for weak sub and super solutions, but for a different range of $p$. 
This depends on the chosen method for the proof of the Caccioppoli type inequality, see also 
\cite[Remark 1.3]{PP}, and it is a classical feature of local boundedness results of this type, see for instance \cite{WZ-preprint}.

Moreover, this result holds true also under a less restrictive assumption on the lower order coefficients, i.e.
$c, f \in L^q_{loc}(\O)$ and $b \in \left( L^q_{loc} (\O) \right)^{m_0}$
	for some $q >\frac{3}{4}\left( Q+2\right)$ with $\div \, b \ge 0$ in $\O$.
The additional requirement on the sign of the divergence is the price to pay in order to lower the integrability requirement on $b$. Indeed, the non-standard structure of the space $\W$ is responsible for several underlying difficulties while proving a local boundedness result. Among these, we find the impossibility to lower the integrability requirements on the term $b$ up to $\frac{Q+2}{2}$, the 
hypoelliptic counterpart of the parabolic homogeneous dimension $N/2$.

It is our belief that further improvements in the integrability requirements for lower order coefficients
may be obtained by taking advantage of Theorem \ref{sob} in the method of the proof of the $L^2-L^\infty$ estimate for weak sub-solutions.

\subsection{Weak Harnack inequality}
The method of the proof of this result is an extension of the
classical one proposed in \cite{Kruzhkov} 
for the elliptic and parabolic setting, and later on followed by Guerand and Imbert in 
\cite{GI} for the Kolmogorov-Fokker-Planck equation. It is based on the combination of the fact that super-solutions 
to \eqref{defL} expand positivity along times with a suitable covering argument. Note that this method is very convenient for the study of the weak regularity theory, because it only relies on the functional structure of the space $\W$ and on the non-Euclidean geometrical setting. 
To the best of our knowledge, 
the following weak Harnack inequality is the only available result of this type for solutions to \eqref{defL} in the framework $\W$. Moreover, we underline that our statement holds true for solution or super-solutions depending on the sign of $c$. 
This is mainly due to the method of proof followed in \cite{AR-harnack}, and the extension to this result to super-solutions without any sign assumption on $c$ is still an open problem.
\begin{theorem}[Weak Harnack inequality]
	\label{weak-harnack}
	Let $R_0>0$.
	Let $\Q^0=B_{R_0}\times B_{R_0}\times \ldots \times B_{R_0} \times(-1,0] $ and let $u$ be a non-negative weak 
	solution to $\L u = f$ in $\O\supset \Q^0$ under 
	assumptions \textbf{(H1)}-\textbf{(H3)}. Then we have
\begin{eqnarray*}
\left(\int_{\Q_-}u^p\right)^{\frac{1}{p}}\leq C\left(\inf_{\Q_+} u + \Vert f \Vert_{L^q(\Q^0)}\right),
\end{eqnarray*}
where $\Q_+=B_{\omega}\times B_{\omega^3}\times\ldots \times B_{\omega^{2\kappa+1}}\times (-\omega^2,0]$ and $\Q_-=B_{\omega}\times B_{\omega^3}\times\ldots \times B_{\omega^{2\kappa+1}}\times (-1,-1+\omega^2]$. Moreover, the constants $C$, $p$, $\omega$ and $R_0$ only depend on the homogeneous dimension $Q$ defined in \eqref{hom-dim}, $q$ and on the ellipticity constants $\lambda$ and $\Lambda$. 
Additionally, if the term $c$ is of positive sign, the statement holds true also for non-negative super-solutions to \eqref{defL}.
\end{theorem}

\subsection{Applications of the Harnack inequality}
It is widely known that invariant Harnack inequalities are one of the most powerful tools in regularity theory. In this subsection, we 
briefly discuss some of the most important applications of Theorem \ref{harnack-thm}, when considering Kolmogorov operators with rough coefficients.

First of all, invariant Harnack inequalities can be used to construct Harnack chains. A Harnack chain connecting any two given 
starting point $z_0$ and ending point $z_k$ of our domain is a set $\{ z_0, z_1, \ldots, z_k \}$ of points of our domain such that there exist $k$ positive
constants $C_1, \ldots, C_k$ for which 
\begin{equation*}
	u(z_j) \le C_j u(z_{j-1}), \quad j=1, \ldots, k
\end{equation*}
for every non-negative solution to $\L u = f$ in $\O$. Hence, a Harnack chain is a set of points through which the quantitative information provided by a Harnack
inequality is able to travel. This tool had widely been employed over the years, especially in combination with techniques from control theory and optimization, for the 
study of qualitative properties of classical solutions to $\L u = f$, where the regularity of the coefficients was assumed to be equal, or better than H\"older continuous,
see for instance \cite{LP}. 

Ever since Theorem \ref{harnack-thm} had been proved, various qualitative results regarding bounds for the weak fundamental solution \cite{AR-funsol, LPP}
where established adapting those techniques to the rough coefficients case. Moreover, we recall the work \cite{AEP}, where the authors prove a geometric characterization 
for the set where the Harnack inequality holds true and strong maximum principle in the Kolmogorov-Fokker-Planck case. Note that, since the proof of this two latter results 
is based on Harnack chains, control theory and an invariant Harnack inequality, their proof straightforwardly applies also to the more general case \eqref{defL} of our interest.

\section{Applications to Physics \& Economics}
\label{applications}
In this section, we provide the reader with a motivation for the study of the class of operators \eqref{defL}, by illustrating some applications of Kolmogorov equations (and associated regularity results) to 
real life problems arising in various research fields, such as Economics and Physics. In particular, we first address applications to option pricing, with a specific focus on American and Asian options. We then present new results regarding a relativistic Kolmogorov-Fokker-Planck equation. Finally, we analyze the link between the linear Kolmogorov-Fokker-Planck equation \eqref{kfp} and the Boltzmann equation.

\subsection{American and Asian options}
In mathematical finance, equations of the form \eqref{defL} appear in various models for the pricing  
of 
financial instruments, such as Asian and American options (cf., for instance, \cite{Barucci, PascucciBook}), as well as in the theory of stochastic utility \cite{Antonelli} and stochastic volatility models \cite{Hobson}. Here, we focus on applications of \eqref{defL} to Asian and American options, and we present an overview of the most recent results in this direction. For a more comprehensive analysis of applications of Kolmogorov operators $\L$ to finance and stochastic theory we refer to the monograph \cite{PascucciBook} by Pascucci. 

Asian options are a family of \emph{path-dependent derivatives}, whose payoff depends on the average of the underlying stock 
price over a certain time interval. In the Black \& Scholes framework, the price of the underlying Stock $S_t$ and of the bond $B_t$ are described by the processes
\begin{equation*}
 S_t = S_0 e^{\mu t + \sigma W_t}, \qquad B_t = B_0e^{rt}, \qquad 0 \le t \le T, 
\end{equation*}
where $\mu, r, T$, and $\sigma$ are given constants. In this setting, the price $(Z_t)_{0 \le t \le T}$ of a continuous path-dependent option is a function $Z_{t} = Z(S_{t},A_{t},t)$ that depends on the \emph{price of the stock} $S_t$, on the \emph{time to maturity} $t$ and on an \emph{average} $A_t$ of the stock price
\begin{equation*} \label{At}
   A_{t} \, = \, \int \limits_{0}^{t} f(S_{\t}) \, \dd\t, \qquad t \in [0, T],
\end{equation*}
and it is computed by solving the following Cauchy problem 
\begin{equation} \label{PDE1}
	\begin{cases}
		\tfrac12 \s^{2} S^{2} \tfrac{\p^{2} Z}{\p S^{2}} + f(S) \tfrac{\p Z}{\p A} + r \left( S \tfrac{\p Z }{\p S} - Z \right) +
		 \frac{\p Z}{\p t} = 0 & (S, A, t) \in \R^{+} \times \R^{+} \times (0, T), \\
		Z(S,A,T) = \phi (S, A)   &(S, A) \in \R^{+} \times \R^{+}.
	\end{cases}
\end{equation}
We remark that the initial data $\phi$ in \eqref{PDE1} corresponds to the pay-off of the option. Moreover, depending on the choice of the function $f(S)$, we find a different Kolmogorov type equation with locally H\"{o}lder continuous coefficients. In \cite{AMP} the first author, Muzzioli and Polidoro, prove the existence and uniqueness of the fundamental solution associated to Geometric Average Asian Options, i.e. when $f(S)=\log(S)$, and Arithmetic Average Asian Options, i.e. when $f(S)=S$ in this framework.
 
On the other hand, an American option with pay-off $\psi$ is a contract which grants the holder the right to receive the payment of the sum $\psi(X_t)$ at a chosen time $t \in [0,T]$, where $X=\left(X_t^x \right)$ is a $N$-dimensional diffusion process which solves the stochastic differential equation
\begin{equation*}
d X_t^x=B X_t^x dt + \sigma(t,X_t^x)d W_t,
\end{equation*}
with $X_{t_0}^{t_0,x}=x$
for $(x,t_0)\in \R^{N}\times [0,T]$ and where, as usual, $(W_t)_{t \geq 0}$ denotes a $n$-dimensional Wiener process, $1 \leq n \leq N$.

In particular, equation \eqref{defL} is relevant in connection to the problem of determining the arbitrage-free price of American options.
Indeed, there are significant classes of American options, whose corresponding diffusion process $X$ is associated to Kolmogorov-type operators which are not uniformly parabolic and are of the kind \eqref{defL}. Two such examples are given by American style options (c.f. \cite{Barucci}) and by the American options priced in the stochastic volatility introduced in the article \cite{Hobson}. Moreover, in virtue of the classical arbitrage theory (see, for instance {\cite{PascucciBook}}), the arbitrage free price at time $t$ of the American option when assuming the risk-free interest rate is zero is given by the following optimal stopping problem
\begin{equation*}\label{optimal}
u(x,t)=\sup_{\tau \in [t,T]}E \left[ \psi \left( X_\tau^{t,x} \right) \right],
\end{equation*}
where the supremum is taken over all stopping times $\tau \in [t,T]$ of $X$. In \cite{Pascucci}, it is proved that the function $u$ in \eqref{optimal} is a solution to the obstacle problem
\begin{align}
		\label{obstaclePP}
		\begin{cases}
			\max \lbrace  \L u -f, \psi -u\rbrace \qquad &(x,t) \in \R^{N} \times [0,T] \\
			u(x,t) = g  \qquad &(x,t) \in \R^{N} \times \lbrace 0 \rbrace,
		\end{cases}
	\end{align}
where $\L$ is the operator \eqref{defL} in trace form, the obstacle $\psi$ corresponds to the pay-off of the option and it is a Lipschitz continuous function in $\Ob$ satisfying a weak convexity condition with respect to the variables $x_1, \ldots, x_n$ (see \cite[Assumption H4.]{obstacle-ex}).

In virtue of its importance in finance, the mathematical study of the obstacle problem \eqref{obstaclePP} was already initiated in the papers \cite{obstacle-ex, obstacle-2, obstacle-3}. More precisely, the main result of \cite{obstacle-ex} is the existence of a strong solution to problem \eqref{obstaclePP} in certain bounded cylindrical domains and in the strips $\R^{N} \times [0,T]$ through the adaptation of a classical penalization technique. On the other hand, the main purpose of papers \cite{obstacle-2, obstacle-3} is to prove some new regularity results for solutions to \eqref{obstaclePP}. In particular, \cite{obstacle-2} concerns the optimal interior regularity for solutions to the problem \eqref{obstaclePP}, while \cite{obstacle-3} contains new results regarding the regularity near the initial state for solutions to the Cauchy-Dirichlet problem and to \eqref{obstaclePP}. 

However, all the results contained in \cite{obstacle-ex, obstacle-2, obstacle-3} only hold true for strong solutions and continuous obstacles satisfying the aforementioned convexity condition. Only very recently, the first and the third author initiated in \cite{AR-obstacle} the study
of obstacle problems associated to Kolmogorov operators in a more general and natural setting, i.e. by considering weak solutions to the obstacle problem related to
 \begin{equation*}\label{fokkerplanck}
\K u (v,x,t):=\nabla_v \cdot \left( A(v,x,t) \nabla_v u (v,x,t) \right)   + v \cdot \nabla_{x} u(v,x,t) - \partial_t u (v,x,t),
\end{equation*} 
in the functional space $\W$ introduced in Section \ref{preliminaries-f}. 
Specifically, in a standard manner (see 
\cite[Chapter 6]{SK}), in \cite{AR-obstacle} it is assumed that obstacle $\psi$ and boundary data $g$ inherit the same regularity of the function 
$u$, namely $\psi \in \W(\O_v \times \O_{xt})$ and $g \in \W(\O_v \times \O_{xt})$, where $\O := \O_v \times \O_{xt}$ is a subset of $\R^{2n+1}$ satisfying the following assumption:
\begin{itemize}
	\item[(D)] $\O_v \subset \R^n$ is a bounded Lipschitz domain 
	and $\O_{xt} \subset \R^{n+1}$ is a bounded domain with 
	$C^{1,1}$-boundary, i.e. $C^{1,1}$-smooth with respect to the transport operator $Y$ as well as $t$.
	\end{itemize}
Thanks to this assumption, it is possible to introduce the outward-pointing unit normal $N$ to $\O_{xt}$ and to classically define the 
	Kolmogorov boundary of the set $\Omega$ as
	\begin{equation*}
		\p_K ( \O_v  \times \O_{xt} ) := 
		\left( \p \O_{v} \times \O_{xt} \right) \cup 
		\left\{ (v,x,t) \in \overline \O_v \times \p \O_{xt} \, | \, \, 
		(v, -1) \cdot N_{xt} > 0 \right\},
	\end{equation*}
which serves in the context of the operator $\L$ as the natural hypoelliptic counterpart of the parabolic boundary considered in the context of Cauchy-Dirichlet 
problems for uniformly parabolic equations. 
In comparison with \cite{obstacle-ex, obstacle-2, obstacle-3, Pascucci}, in \cite{AR-obstacle} the authors weaken the regularity assumptions on the right-hand 
side by considering  $f \in L^2(\O_{xt}, H^{-1}(\O_v))$. Furthermore, the following more general obstacle problem than \eqref{obstaclePP} is considered
	\begin{align}
		\label{obstacle}
		\begin{cases}
			\K u(v,x,t) = f(v,x,t) \qquad &(v,x,t) \in \O, \\
			u(v,x,t) \ge \psi(v,x,t)  \qquad &(v,x,t) \in \O, \\
			u(v,x,t) = g  \qquad &(v,x,t) \in \p_K \O,
		\end{cases}
	\end{align}
	where the boundary condition needs to be considered as attained in the sense of traces, the obstacle condition holds in $\W(\O_v \times \O_{xt})$
	and 
	\begin{equation*}
		\psi \le g \, \, \, \text{on } \p_K (\O_{v} \times \O_{xt})
		\quad \text{in } \, \, \W(\O_v \times \O_{xt}).
	\end{equation*}
Finally, the condition $\K u(v,x,t) = f(v,x,t) $ for $(v,x,t) \in \O$ appearing in \eqref{obstacle} needs to be interpreted as stating that
\begin{align}\label{weakformu}
			0 = 
			 \iiint \limits_{\O_{v} \times \O_{xt}} A(v,x,t) \nabla_v u \cdot \nabla_v \phi 
				\, \dd v \, \dd x \, \dd t +
			     \iint \limits_{\O_{xt}} \langle f (\cdot, x,t) - Y u (\cdot, x,t) | \phi(\cdot, x,t) \rangle 
			     \, \dd x \, \dd t     
		\end{align}
		for every $\phi \in L^2(\O_{xt}, H^1_c(\O_v))$ and where $\langle \cdot | \cdot \rangle$ is the standard duality pairing
		in $H^{-1}(\O_v)$.		
Then, the main result of \cite{AR-obstacle} is the following.
\begin{theorem}
		\label{main-thm}
		Let us assume that the diffusion matrix $A$ in \eqref{fokkerplanck} satisfies the ellipticity condition in {\bf (H1)} for $m_0=n$. 		
		Let $f \in L^2(\O_{xt}, H^{-1}(\O_v))$ and
		$g, \psi \in \W(\O_v \times \O_{xt})$, where $\O$ is a subset of $\R^{2n+1}$ satisfying assumption (D).
		Then there exists a unique weak solution $u \in \W(\O_{v} \times \O_{xt})$
		in the sense of equation \eqref{weakformu}
		to the obstacle problem \eqref{obstacle}. 
		Moreover, there exists a constant $C$, which only depends on $d$ and on $\O_{v} \times \O_{xt}$, such that
		\begin{equation*}
		\| u \|_{\W( \O_{xt} \times \O_{v})} \leq C 
		\left( \| g \|_{\W( \O_{xt} \times \O_{v})}+\| f \|_{L^2(\O_{xt}, H^{-1}(\O_v))}\right).
		\end{equation*}				
	\end{theorem}
We eventually point out that with \cite{AR-obstacle} the authors aim at initiating the study of the obstacle problem \eqref{obstacle} in the framework of Calculus of Variations, 
by rewriting the problem of finding a solution to \eqref{obstacle} as that of finding a null minimizer of the functional
\begin{multline}
		\label{functional-intro}
		 \inf \Big\{  \,\, \iiint \limits_{\O_v \times \O_{xt}} \frac12 
		\left( A \left( \nabla_v u - \mathfrak{J} \right) \right) \cdot
					\left( \nabla_v u - \mathfrak{J} \right) \, \dd v \, \dd x \, \dd t \,: \\
		  \mathfrak{J} \in \left( L^2(\O_{xt};L^2(\O_v)) \right)^n \, \, s. \, t. \, \, \nabla_v \cdot \left( A \mathfrak{J} \right) = f - Yu \Big\},
	\end{multline}
It is clear that the infimum in \eqref{functional-intro} is non-negative	and that, given a solution $u$ to \eqref{obstacle}, if we choose $\mathfrak{J}=\nabla_v u$, then \eqref{functional-intro} vanishes at $u$. Moreover, it is easy to show that the functional in \eqref{functional-intro} is uniformly convex and attains its minimum at zero. Finally, we observe that the functional justifies the definition of functional kinetic space given in Section \ref{preliminaries-f}. However, it is still an open problem whether it is possible to employ classical tools from Calculus of Variations to study the variational problem associated to functional \eqref{functional-intro}.

\subsection{Relativistic Fokker-Planck equation}
As pointed out in Example \ref{kfp-ex}, the class of Kolmogorov operators \eqref{defL} arises in many Physical applications. For this reason, we consider equation \eqref{kfp-nofriction} in the framework of special relativity, namely
%
\begin{equation} \label{oprel}
    \L u (p,x,t) =  {\sqrt{|p|^{2} + 1}}\, \nabla_p \cdot \left( D \, \nabla_p u\right) 
    - p \cdot \nabla_x u - \sqrt{|p|^{2} + 1} \, \p_t u = 0,
\end{equation}
where $(p,x,t) \in \R^{2n+1}$ and $D$ is the \textit{relativistic diffusion matrix} given by 
\begin{equation*}
 D=\frac{1}{\sqrt{|p|^{2} + 1}}\left( \mathds{I}_n+ p\otimes p \right).
\end{equation*}
Here and in the following, $\mathds{I}_n$ denotes the $n \times n$ identity matrix and 
$$
p\otimes p = \left( p_i p_j \right)_{i,j=1, \dots, n}.
$$ 
In this context, a solution $u=u(p,x,t)$ to \eqref{oprel} denotes the density of particles in the phase space with momentum $p$ and position $x$, at time $t$.
Equation \eqref{oprel} is a generalization which agrees with special relativity of the frictionless kinetic Fokker-Planck equation \eqref{kfp-nofriction}, as 
the \emph{relativistic velocity}
\begin{equation*}
v =\frac{p}{\sqrt{|p|^2+1}},
\end{equation*} 
clearly satisfies
\begin{equation*}
	\left\vert \frac{p}{\sqrt{|p|^{2} + 1}}\right\vert < 1 \quad \text{for every} \ p \in \R^n,
\end{equation*}
in accordance with the relativity principles\footnote{Here, we adopt a natural unit system with $c=1$, where $c$ is the speed of light.}.

Operator $\L$ in \eqref{oprel} serves as a suitable relativistic version of $\K_0$ in \eqref{kfp-nofriction} as it preserves some relevant properties which hold in the non-relativistic setting (we refer the reader to \cite{AC} for a more rigorous derivation of equation \eqref{oprel}). In particular, operator $\L$ satisfies the relativistic analogue of property \eqref{inv-K}, i.e. it is invariant under Lorentz transformations. Taking advantage of this property, in \cite{relativistico} Polidoro and two of the authors costructed the invariance group of $\L$ by defining the composition law as follows
\begin{align}\label{group-law-lorentz}
	(p_0,x_0,t_0) \circ_{\mathcal{L}} (p,x,t)= \Big( &p\sqrt{|p_0|^2+1}+p_0 \sqrt{|p|^2+1}, \\ \nonumber
				&x_0+x\sqrt{|p_0|^2+1}+p_0t, t_0+t\sqrt{|p_0|^2+1}+p_0 \cdot x  \Big ).
\end{align}
We remark that for small velocities $\sqrt{1+|p_0|^2} \approx 1$ and therefore \eqref{group-law-lorentz} becomes precisely the non-relativistic composition law \eqref{galilean} for variables $p$ and $x$. The introduction of the composition law \eqref{group-law-lorentz} and consequently of an appropriate non-Euclidean structure on the space $\R^{2n +1}$ significantly simplifies the study of the regularity of operator $\L$.

As with the non-relativistic operator in \eqref{kfp-nofriction}, $\L$ is a strongly degenerate differential operator, since only second order derivatives with respect to the momentum variable $p \in \R^n$ appear. However, the first order part of $\L$ induces a strong regularizing property, namely $\L$ is hypoelliptic.  
Indeed (see \cite[Appendix A]{relativistico}) we can write $\L$ as a \textit{sum of squares plus a drift term}
\begin{equation*} 
	\L := \sum_{j=1}^n X_j^{2} + X_{n+1},
\end{equation*}
with
\begin{equation*} 
 X_j = \sum_{k=1}^{n} \left( \delta_{jk} +\tfrac{p_j p_k}{1 + \sqrt{|p|^{2} + 1}}\right) \tfrac{\p  }{ \p p_k}, 
 \quad j=1, \dots, n, \quad  \text{and} \quad X_{n+1} = \sum_{k=1}^{n} c_k(p) X_k - Y, 
\end{equation*}
where $c_1, \dots c_n$ are smooth functions and 
\begin{equation*} 
		Y =  p \cdot \nabla_{x} + {\sqrt{|p|^{2} + 1}} \, \tfrac{\p}{\p t}.
\end{equation*}
Moreover, as shown in \cite[Appendix A]{relativistico}, $\L$ satisfies H\"ormander's rank condition \eqref{hormander} at every point $(p,x,t) \in \R^{2n+1}$.

It is then natural to consider the relativistic analogous of \eqref{langevin}, namely
\begin{equation}\label{RSDE}
\begin{cases}
    & P_{s} = p_{0} + \sqrt{2} \int \limits_0^s \sqrt{P_\tau^2+1}\, dW_{\tau}, \\
	& X_{s} = x_{0} + \int \limits_{0}^{s} P_{\tau} \dd\tau,\\
	& T_{s} = t_{0} + \int \limits_{0}^{s} \sqrt{P_\tau^2+1}\, \dd\tau,
\end{cases}
\end{equation}
where the third component is the time, which is not an absolute quantity in the relativistic setting. It is clear that \eqref{oprel} is the relativistic deterministic equation describing the density of the stochastic process \eqref{RSDE}. The main result of \cite{relativistico} provides us with a lower bound for such a density function and can be stated as follows for $(p,x,t)\in \R^3$.
\begin{theorem} \label{boundsL}
Let $\Gamma$ be the fundamental solution of $\L$ in \eqref{oprel}. Then for every $T > 0$ there exist three positive constants $\theta, c_T, C$ with $\theta < 1$, such that 
\begin{equation*}
	\Gamma (p_1,x_1,t_1;p_0,x_0,t_0) \ge \frac{c_T}{(t_1 - t_0)^{2}} \exp \left\{ - C \, \Psi \left(p_1, x_1, y_1; p_0, x_0, \theta^2t_0 + (1 - \theta^2)t_1 \right)  \right\}
\end{equation*}
for every $(p_0,x_0,t_0), (p_1,x_1,t_1) \in \R^3$ such that $0 < t_1 - t_0 < T$. The constants $\theta$ and $C$ only depend on $\L$, while $c_T$ also depends on $T$. Moreover, $\Psi$ is the \emph{value function} of a suitable optimal control problem (see \cite[Section 3]{relativistico}).
\end{theorem}
We point out that Theorem \ref{boundsL} only constitutes a first step towards developing a systematic and more comprehensive study of $\L$ in its appropriate framework of PDEs theory. Indeed, the purpose of \cite{relativistico} is to propose an approach that may lead to
various developments in Stochastic and Kinetic Theory, where the final aim is to extend the classical theory considered in \cite{APsurvey} to the relativistic case.

\subsection{Boltzmann equation}\label{Boltzmann}
Among the most classical applications of Kolmogorov operators to Physical modeling we find the Boltzmann operator, for which \eqref{kfp} 
represents a linearization. In ~$\R^{2n+1}$, the \ap{\it Boltzmann equation}'' reads as
\begin{equation} \label{boltzmann}
	\partial_t u(v,x,t) +v \cdot\nabla_x u(v,x,t) = \Lc(u), \quad \textrm{for}~(v,x,t) \in B_1\times B_1 \times (-1,0],
\end{equation}
where the function~$u\equiv u(v,x,t)$ is defined for any~$(v,x,t) \in \R^n \times B_1 \times (-1,0]$ and the nonlocal {\it collisional} operator~$\Lc$ is given by
\begin{equation*}
\Lc(u) := \iint_{\R^n \times  \mathds{S}^{n-1} } \big( u(v_{*}^{'}) u(v^{'}) - u(v_{*}) u(v) \big) 
	K\left( |v-v_{*}|, \cos \theta \right) \, {\rm d}v_{*} {\rm d} \sigma,
\end{equation*}
where~$v_{*}^{'}$ and~$v^{'}$ are computed using~$v$,~$v_{*}$ and~$\sigma$ via the following formulae
\begin{align*}
	v'=\frac{v+v_*}{2} + \frac{ |v-v_*| }{2}\sigma \quad \text{and} \quad 
	v'_*=\frac{v+v_*}{2} - \frac{ |v-v_*| }{2}\sigma,
\end{align*}
and where~$\theta$ is the angle measuring the deviation between~$v$ and~$v'$, whose cosinus is defined as
\begin{align*}
	\cos \theta := \frac{v+v_{*}}{|v-v_{*|}} \sigma.
\end{align*}
Equation~\eqref{boltzmann} describes the dynamics of a dilute gas. Indeed, the transport term~$\big(\partial_t + v \cdot \nabla_x\big)$ on the left-hand side of~\eqref{boltzmann}
describes the fact that particles travel in straight lines when no external force is applied, while
the diffusion term~$\Lc$ expresses the fluctuations in velocity arising from particles interactions. In this setting, the solution~$u$ represents the density of particles with position~$x$ and velocity~$v$ at time~$t$, characterizing the state of the gas in a statistical way. Furthermore, the couples~$v$,~$v_*$, and~$v^{'}$,~$v^{'}_*$ represent pre and post collisional velocities of two distinct particles, respectively. The rate by which such particles switch from initial velocities~$v$,~$v_*$ to~$v^{'}$,~$v^{'}_*$ after the collision is given by the kernel~$K$, which is usually know as {\it collision kernel}. We refer the interested reader to~\cite{IS-weak,IS} and the refernces therein for some regularity and related properties of solutions to~\eqref{boltzmann} when the collision kernel does coincide with
\begin{equation*}\label{collision_krn}
		K(r, \cos \theta) = r^\alpha b(\cos \theta), \qquad \textrm{with}~b(\cos \theta) \approx | \sin (\theta/2) |^{-(n-1)-2s},
\end{equation*}
where~$\alpha > -n$ and~$s \in (0,1)$.

\section{Nonlinear nonlocal Kolmogorov-Fokker-Planck equations}
\label{nnkfp}
In recent years  an increasing interest has been focused on the study of fractional powers of nonlinear operators, 
not only because of their appearance in many concrete models in Physic, Biology and Finance, but also because of their challenging mathematical intrinsic 
nature, mainly due to the contemporary presence of a nonlinear and a nonlocal behaviour (see for example~\cite{BV15} and the references therein).

Naturally, presenting a comprehensive treatment of the related literature is beyond the scopes of the present overview, but we still mention~\cite{DKP14,DKP16}, where amongst other results the authors prove the De Giorgi-Nash-Moser theory for (elliptic) nonlinear fractional operators modeled on the fractional $p$-Laplacian; see also the survey article~\cite{Pal18} and the references therein.
Moreover, we recall that some notable results were established even in the parabolic setting, e.~\!g.~\cite{DZZ21} for the De Giorgi-Nash-Moser theory in the superlinear case when~$p\ge2$, and the recent papers~\cite{APT22,Liao22} for  H\"older regularity for any value of the exponent~$p \in (1,\infty)$.

Hence, for the community it had been quite natural to begin considering ultraparabolic equations whose diffusion part is modeled on the fractional $p$-Laplacian, with the aim of extending the study of the regularity theory to this case. The reference model is now a nonlinear version of the Boltzmann equation~\eqref{boltzmann} defined as
\begin{equation}\label{nonloc_fokker_planck}
	\partial_t u(v,x,t) +v \cdot \nabla_x u(v,x,t)+ \mathcal{L}_{K}(u) = f(v,x,t,u),  \qquad \mbox{in}~\R^{2n+1}.
\end{equation}

In the previous display the inhomogeneity~$f: \rnn \times \mathds{R} \to \mathds{R}$ is a Carath\'eodory function satisfying the following growth condition
$$
|f(v,x,t,u)| \le c_{\rm o} |u|^{\gamma-1} + h(v,x,t),  \qquad \textrm{for a.~\!e.}~(v,x,t,u) \in \rnn\times \R,
$$
for some a positive constant~$c_{\rm o}$,~$\gamma \in (1,p]$ and a given function~$ h \in L^{\infty}_{\rm loc}(\rnn)$.

The leading operator~$\Lc$, representing the \ap diffusion " in the velocity variable, is an integro-differential operator of differentiability order~$s \in (0,1)$ and integrability exponent~$p\in (1,\infty)$, whose explicit expression is given by
\begin{equation}\label{nonloc_diff}
	\Lc (u)(v,x,t) := \lim_{\eps \to 0^+}\int_{\R^n \smallsetminus B_\eps(v)} \Ac u (v,w,x,t)K(v,w) \,\dd w,  
\end{equation}
where for the sake of readability 
$$
\Ac u (v,w,x,t) := |u(v,x,t)-u(w,x,t)|^{p-2}(u(v,x,t)-u(w,x,t)).
$$
The operator~$\Lc$ is driven by its nonsymmetric measurable kernel~$K : \mathds{R}^n \times \mathds{R}^n \rightarrow [0,+\infty)$, which satisfies, for some~$0<\lambda \le \Lambda$, the following bounds
\begin{equation}\label{main_krn}
	  \lambda|v-w|^{-n-sp} \le K(v,w) \le \Lambda |v-w|^{-n-sp}, \qquad \textrm{for a.~\!e.}~v,w \in \mathds{R}^n.
\end{equation}
We remark that operator~\eqref{nonloc_diff} is compatible with the Boltzmann diffusion term in~\eqref{boltzmann} in the linear case, when~$p=2$, and while acting on nonnegative functions.
Indeed, conditions~\eqref{main_krn} are compatible  with the Boltzmann collision kernel~\eqref{collision_krn} if the following macroscopical physical quantities 
 \begin{eqnarray*}
 M(x,t) &:=& \int_{\R^n} u(v,x,t) \, \dd v \qquad \qquad \textrm{(mass density)},\\
 E(x,t) &:=& \int_{\R^n} u(v,x,t)|v|^2 \, \dd v \qquad \, \, \textrm{(energy density)},\\
 H(x,t) &:=& \int_{\R^n} u \ln u(v,x,t) \, \dd v \qquad \, \textrm{(entropy density)},
 \end{eqnarray*}
are bounded; see for instance the statement of~\cite[Assumption~1.1]{IS-weak}. Thus, equation~\eqref{nonloc_fokker_planck} can be seen as a  generalization of~\eqref{boltzmann} and, for this reason, there is interest in studying qualitative {and regularity} properties {of its related weak} solutions.

	In the linear case when~$p=2$ many remarkable results are available, e.~\!g. H\"older regularity~\cite{Sto19},~$L^p$-estimates~\cite{CZ18,HMP19}, hypoelliptic regularity~\cite{Li14}, existence of weak solutions ~\cite{WD17} and existence, uniqueness and regularity of solutions in the viscosity sense~\cite{Imb05}. 
		
	Furthermore, we recall the recent paper~\cite{Loh22} where the author proves H\"older continuity together with some weak Harnack-type inequalities for nonnegative and apriori bounded weak solutions to fractional Kolmogorov-Fokker-Planck equations.

	However, in the more general setting when a $p$-growth is involved, the theory is completely lacking and the results contained in Theorem~\ref{loc_bdd_sol} below, proved by two of the authors in~\cite{AP-boundedness}, would be the first, serving also as first attempt in proving that weak solutions to~\eqref{nonloc_fokker_planck} enjoy some expected local properties.
	
	Moreover, Theorem~\ref{loc_bdd_sol} is new even in the linear case when~$p=2$ where boundedness is usually taken as assumption in regularity theory; see for instance the aforementioned~\cite[Theorem~1.2]{Loh22}.

 Throughout this section, we will recall some helpful properties about the underlying geometrical and fractional functional setting suitable for the study of \eqref{nonloc_fokker_planck}. Then, we recall the statement of the boundedness estimates for weak solutions to \eqref{nonloc_fokker_planck} {for any value of the integrability exponent}~$p \in (1, + \infty)$ providing the reader with a comparison with existing results in literature and with related open problems and further developments.

\subsection{Geometric and functional setting}
As in the previous sections, we denote with~$z=(v,x,t)$ points of~$\R^{2n+1}$ and  
we define a family of anisotropic dilations~$(\delta_r)_{r>0}$ on~$\R^{2n+1}$ in the following way
     $$
     	\delta_r =\textrm{diag}(r\mathds{I}_n, r^{1+sp}\mathds{I}_n,r^{sp}t), \quad \forall r >0.
    $$ 
Firstly, we observe that when $s$ tends to $1$ we recover the family of dilations introduced in \eqref{gdil} for the linear Kolmogorov-Fokker-Planck operator. Moreover, analogously as in the linear case, equation~\eqref{nonloc_fokker_planck} is homogeneous of degree~$2$ with respect to the 
dilation group~$(\delta_r)_{r>0}$ just introduced. In addition, we endow~$\R^{2n+1}$ with the product law
     \begin{equation}\label{group_law}
     	 z_0 \circ z = (v_0+v,x_0+x+tv_0,t_0+t), \qquad \forall z_0=(v_0,x_0,t_0) \in \R^{2n+1},
     \end{equation}
that is the same Galilean change of variables \eqref{galilean} considered in the linear Kolmogorov-Fokker-Planck case. This is due to 
     the fact that, geometrically speaking, the family of translations is strongly connected to the shape of the transport operator, which in this case agrees with 
     the one considered in \eqref{kfp}, i.e. given by $\p_t + v \cdot \nabla_x$.
     In this way,~$\mathds{K}:=(\R^{2n+1}, \circ)$ is a Lie group with identity element~$e:=(0,0,0)$ and inverse given by
     $$
     z^{-1}:=(-v,-x+tv,-t), \qquad \forall z=(v,x,t) \in \R^{2n+1},
     $$
     and, for sufficiently regular functions, equation~\eqref{nonloc_fokker_planck} is invariant with respect to the Lie product~\ap $\circ$''.
     
    Then, as it is done in the local framework, we introduce
     two families of fractional kinetic cylinders; one defined  starting from the aforementioned dilations and product law
     \begin{equation}
     	\label{Q-classico}
     Q_r(z_0) := \big \{ z:~| v- v_0 | < r,~| x - x_0 - (t - t_0) v_0| < r^{1 + sp},~t_0 - r^{sp} < t < t_0 \big \},
     \end{equation}
    and the other defined via a ball representation formula
     \begin{equation}\label{frac_kin_cylinder}
     	\Q_r(z_0) := B_r(v_0) \times U_r(x_0, t_0) := B_r(v_0) \times B_{r^{1+sp}}(x_0) \times  (t_0-r^{sp},t_0).
     \end{equation}
    
     All the results below are stated for cylinders defined as in~\eqref{frac_kin_cylinder} via ball representation. These cylinders turn out to be equivalent to the ones defined in~\eqref{Q-classico}.  Indeed, the following Lemma holds true; see~\cite[Lemma~2.2]{AP-boundedness}.
     \begin{lemma}
     	For every~$z_0 \in \R^{2n+1}$ and every~$r>0$, there exists a positive constant~$\vartheta$ such that 
     	\begin{equation*}
     		Q_{\frac{r}{\vartheta}} (z_0) 
     		\subset \Q_{r}(z_0) 
     		\subset Q_{r \vartheta} (z_0) .
     	\end{equation*}
     \end{lemma}

     We now introduce the fractional functional setting.  For any~$s \in (0,1)$ and~$p \in (1,\infty)$, we define the fractional Sobolev space as
     $$
     W^{s,p}(\R^n):= \Big\{ g \in L^p(\R^n): \iint_{\R^n \times\R^n} \frac{|g(v)-g(w)|^p}{|v-w|^{n+sp}}\, \dd v \, \dd w <\infty\Big\},
     $$
     and we endow it with the norm
     $$
     \|g\|_{W^{s,p}(\R^n)}:= \|g\|_{L^p(\R^n)} + \left(\iint_{\R^n \times\R^n} \frac{|g(v)-g(w)|^p}{|v-w|^{n+sp}}\, \dd v \, \dd w\right)^\frac{1}{p},
     $$
     turning it into a Banach space. In a similar way one can define the fractional Sobolev space~$W^{s,p}(\Ov)$ on an open and bounded set~$\Ov \subset \R^n$. Finally, we denote with~$W^{s,p}_0(\Ov)$ the closure with respect to~$\|\cdot\|_{W^{s,p}(\Ov)}$ of~$C^\infty_c(\Ov)$.

     The difficulty in studying operators such as~$\Lc$ lies in their own definition. Indeed, we have to deal both with their nonlinear structure as well as with their nonlocal behaviour. Thus, as natural in the fractional framework, a tail-type contribution needs to be taken in consideration in order to carefully control the long-range interactions that naturally arise in the study of fractional problems.
     In the Kolmogorov-Fokker-Planck setting, the kinetic nonlocal tail mainly considers long-range contributions associated with the 
     velocity variable appearing in the kernel.
     
     \begin{defn} \label{tail-def}
     	Let~$z_0\in \O:= \Ov \times \Ox \times (t_1,t_2) \subset \R^{2n+1}$ and~$r>0$ 
     	be such that~$B_r(v_0)\times U_{2r} (x_0, t_0) \subset  \Ov\times \Ox \times (t_1,t_2)$.
     	Let~$u$ be a measurable function on~$\R^n \times \Ox \times (t_1,t_2)$. Then for any~$r>0$ we define the {\rm kinetic nonlocal tail} of~$u$ 
     	with respect to~$z_0$, and~$r$ as
     	\begin{equation*}
     		\Tl(u;z_0,r):= \Biggl(\,r^{sp} \dashint_{U_{2r}(t_0,x_0)} \int_{\R^n \smallsetminus B_r(v_0)}\frac{|u(v,x,t)|^{p-1}\, \dd v \,\dd x \,\dd t}{|v-v_0|^{n+sp}} \Biggr)^{\frac{1}{p-1}}.
     	\end{equation*}
     	Furthermore, we define the {\rm kinetic nonlocal supremum tail} of~$u$ with respect to~$z_0$ and~$r$ as
     	\begin{equation}\label{tail}
     		\Tl_{\infty}(u;z_0,r):= \Biggl(r^{sp}\sup_{(x,t) \in U_{2r}(x_0,t_0)} \, \int_{\R^n \smallsetminus B_r(v_0)}\frac{|u(v,x,t)|^{p-1}\, \dd v}{|v-v_0|^{n+sp}}\Biggr)^\frac{1}{p-1}.
     	\end{equation}
     \end{defn} 
     The nonlocal tail of a function was firstly introduced in the aforementioned papers~\cite{DKP14,DKP16} in the elliptic setting and subsequently used to derive fine properties of weak solutions to nonlocal operator; see~\cite{Pal18} and the references therein. Various extensions in the parabolic setting can be found in~\cite{APT22,DZZ21,Kim19,Liao22,Strom19,Strom19-2}. Moreover, an analogous quantity has been defined in the sub-Riemannian framework of the $n$-dimensional Heisenberg group~$\mathds{H}^n$ in~\cite{MPPP21,PP21} and, despite its young age, used to prove some regularity result for solution to fractional (possibly nonlinear) equations; see~\cite{GLV23,Pic22}.

With this bit of notation {it is possible to} introduce the definition of weak solution to~\eqref{nonloc_fokker_planck}, {that is a 
function belonging to the space	}
$$
				\W:= \big\{ u \in L^p_{ loc}  (\R^{n+1} ;W^{s,p}(\R^n)) \,
				:(\partial_tu +v\cdot \nabla_x u) \in L^{p'}_{ loc}( \R^{n+1}; ( W^{s,p}(\R^n) )^* ) \big \}
$$
and satisfying \eqref{nonloc_fokker_planck} in the following sense.
	\begin{defn}
		\label{weak-sol-int}
		Let~$\O := \O_v \times \O_x \times (t_1,t_2) \subset \R^{2n+1}$.
		A function~$u \in \W$ is a weak subsolution (resp. supersolution) to~\eqref{nonloc_fokker_planck} in~$\Omega$ if 
		\begin{equation}\label{bdd_sup_tail}
		\Tl_\infty(u_+;z_0,r)<\infty, \qquad {\rm (} \mbox{resp.}~\Tl_\infty(u_-;z_0,r)<\infty{\rm)},
		\end{equation}
		for any~$z_0$ and~$r>0$ such that~$U_{2r} (x_0, t_0) \subset  \Ox \times (t_1,t_2)$ and~$B_r(v_0) \subset \O_v$, and
		\begin{align*} 
			&\int _{\O_x \times (t_1,t_2)} \iint _{\R^n 
				\times \R^n} \Ac  u(v,w,x,t) \left( \phi(v,x,t) 
			- \phi (w,x,t) \right) \, \dd w \, \dd v \, \dd x \, \dd t \nonumber\\
			&\quad  \ge~\text{{\rm(}$\le$, resp.{\rm)}}
			\int  _{\R^n \times \O_x \times (t_1,t_2)}  
			\left(  u\,\partial_t\phi+u\,v\cdot\nabla_x\phi + f( u) \phi   \right)(v,x,t) \, \dd v \, \dd x \, \dd t,
		\end{align*}
		for any nonnegative~$\phi \in L^p(\O_x \times (t_1, t_2); W^{s,p}_0(\O_v))$ such that~$(\partial_t\phi+v\cdot\nabla_x\phi\big) \in L^{p'}( \O_x \times (t_1, t_2); ( W^{s,p}$$(\R^n) $$)^* $$)$ and~$\supp(\phi) \Subset \Omega$. 
		
		We say~$u$ is a weak solution to~\eqref{nonloc_fokker_planck} in~$\Omega$ if 
		\begin{equation}\label{bdd_sup_tail}
			\Tl_\infty(u;z_0,r)<\infty,
		\end{equation}
		for any~$z_0$ and~$r>0$ such that~$U_{2r} (x_0, t_0) \subset  \Ox \times (t_1,t_2)$ and~$B_r(v_0) \subset \O_v$, and
        \begin{align*} 
        	&\int _{\O_x \times (t_1,t_2)} \iint _{\R^n 
        		\times \R^n} \Ac  u(v,w,x,t) \left( \phi(v,x,t) 
        	- \phi (w,x,t) \right) \, \dd w \, \dd v \, \dd x \, \dd t \nonumber\\
        	&\quad  =
        	\int  _{\R^n \times \O_x \times (t_1,t_2)}  
        	\left(  u\,\partial_t\phi+u\,v\cdot\nabla_x\phi + f( u) \phi   \right)(v,x,t) \, \dd v \, \dd x \, \dd t,
        \end{align*}
		for any~$\phi \in L^p(\O_x \times (t_1, t_2); W^{s,p}_0(\O_v))$ such that~$(\partial_t\phi+v\cdot\nabla_x\phi\big)  \in L^{p'}( \O_x \times (t_1, t_2); ( W^{s,p}(\R^n) )^*)$ and~$\supp(\phi) \Subset \Omega$. 
	\end{defn}
	
	Note that, also in this framework, the transport derivative~$(\partial_t+v\cdot\nabla_x\big) $ appearing in Definition~\ref{weak-sol-int} above  needs to be intended in the duality sense, 
	as suggested by our functional setting. Thus, our notation stands for the more formal one: 
	\begin{eqnarray*}
		\int  _{\R^n \times \O_x \times (t_1,t_2)}  u (\partial_t\phi+v\cdot\nabla_x\phi\big)   \, \dd v\,\dd x\,\dd t 
		- \int  _{\O_x \times (t_1,t_2)} \langle \big(\partial_t u+v\cdot\nabla_x u\big) \mid \phi  \rangle \,\dd x\,\dd t ,
	\end{eqnarray*}
	where~$\langle \cdot \,|\, \cdot \rangle$ denotes the standard duality pairing between~$W^{s,p}_0(\R^n)$ and $( W^{s,p}$$(\R^n) $$)^*$.

\subsection{Boundedness estimates}
The results are divided in two case depending on the range of the integrability exponent. 
    
    We start with the superlinear case when~$p\ge2$. The first main result regards the boundedness from above for weak subsolutions, from which we will later on derive the desired $L^\infty$-bound for weak solutions.
	
	\begin{theorem}\label{loc bdd}
		Let~$p \in [2,\infty)$,~$s \in (0,1)$ and let~$u$ be a weak subsolution to~\eqref{nonloc_fokker_planck} in~$\O$ according to Definition~{\rm\ref{weak-sol-int}} and~$\Q_r(z_0) \Subset  \O$. Then, for any~$\delta \in (0,1]$, it holds
		\begin{eqnarray}\label{bdd}
			\sup_{\Q_{\frac{r}{2}}(z_0)}u  &\le&  C\delta^{-\frac{n(p-1)}{sp^2}} \, \max \Big\{\Big(\, \dashint_{\Q_r(z_0)}u_+^p \,\dd v\,\dd x\,dt\Big)^\frac{1}{p},1\Big\} \\
			&&  + \delta \, \Tl_\infty \big( u_+;z_0,\frac{r}{2}\big) +    C\,\|h\|_{L^\infty(\Q_r(z_0))}^\frac{1}{p-1},\nonumber
		\end{eqnarray}
		where~$\Tl_\infty(\cdot)$ in~\eqref{tail} and~$C= C(n,p,s,c_{\rm o}, \gamma, \lambda, \Lambda)>0$.
	\end{theorem}
	
	Moreover, as explained in~\cite[Remark~5.2]{AP-boundedness}, we may adapt the proof of Theorem~\ref{loc bdd} to obtain a lower bound analogous to~\eqref{bdd} for weak supersolutions. Besides, combining this lower bound with~\eqref{bdd}, we get the desired local boundedness whenever~$u$ is a weak solution.
	
	\begin{theorem}\label{loc_bdd_sol}
		Let~$p \in [2,\infty)$,~$s \in (0,1)$ and let~$u$ be a weak 
		solution to~\eqref{nonloc_fokker_planck} in~$\O$ according to Definition~{\rm\ref{weak-sol-int}}. Then,~$u \in L^\infty_{\loc}(\Omega)$.
	\end{theorem}
	
	The{\it singular case} when~$p \in (1,2)$ is more technical and in order to deal with the singularity of the nonlinear term in~\eqref{nonloc_diff} appearing in the kernel~$\Ac$ a further hypothesis also used in the parabolic framework, see e.~\!g.~\cite{DZZ21,HT22}, is needed to prove the desired upper estimate.
	
	\vspace{2mm}
	\begin{itemize}
		\item[\bf(H$_S$)] There exists a sequence~$\{u_\ell\}_{\ell \in \N}$ of bounded weak subsolutions to~\eqref{nonloc_fokker_planck} such that, for any~$z_0$ and~$r>0$ for which~$\Q_r(z_0) \Subset \O$, it holds
		$$
			\Tl_{\infty}\big((u_\ell)_+; z_0, \frac{r}{2}\big) \leq C, \qquad \textrm{for any}~\ell >0,
     	$$
		and
		$$
			u_\ell \to u, \qquad \textrm{in}~L^{2/p}(\Q_r(z_0)).
    	$$
		\end{itemize} 
	
	\begin{theorem}\label{loc bdd2}
		Let~$p \in (1,2)$,~$s \in (0,1)$ and let~$u$ be a weak 
		subsolution to~\eqref{nonloc_fokker_planck} in~$\O$ according to Definition~{\rm\ref{weak-sol-int}},~$\Q_r(z_0) \Subset  \O$
		and let {\bf (H$_S$)} holds. Then, for any~$\delta \in (0,1]$, it holds
		\begin{eqnarray}\label{eq_bdd2}
			\sup_{\Q_{\frac{r}{2}}(z_0)} u  &\le &  C \, \delta^{-\frac{n(p-1)}{sp}} \, \max \Big\{\, \dashint_{\Q_r(z_0)} u_+^{2/p} \, \dd v\,\dd x\,\dd t\, , \, 1\Big\}\\
			&& +\delta\Tl_\infty \big( u_+;z_0,\frac{r}{2} \big) +  	C \, \|h\|_{L^\infty(\Q_r(z_0))}^\frac{1}{p-1} , \nonumber
		\end{eqnarray}
		where~$\Tl_\infty(\cdot)$ in~\eqref{tail} and~$C= C(n,p,s,c_{\rm o}, \gamma, \lambda, \Lambda)>0$.
	\end{theorem}

	 Because of the arbitrary range of the exponent~$p\in (1,\infty)$, we can not establish the desired estimate~\eqref{bdd} by simply applying most of the methods used in the linear framework (e.~\!g. velocity averaging and Fourier transform~\cite{Bouchut} or Dirichlet-to-Neumann's map method~\cite{GT21}). For this reason, starting from a proper fractional Caccioppoli-type inequality, we exploit an iterative scheme which takes in consideration the transport operator~$(\partial_t+v\cdot\nabla_x)$ alongside {the diffusion term}~$\Lc$. Further efforts are also needed in order to deal with the inhomogeneity~$f$ appearing on the righthand side of~\eqref{nonloc_fokker_planck}.
	
	We also notice that in the singular case, when~$p\in(1,2)$, the technical assumption~{\bf (H$_S$)} come into play when proving that the upper bound~\eqref{eq_bdd2} is satisfied uniformly by the approximating sequence~$\{u_\ell\}_{\ell \in \N}$ and subsequently passing to the limit as~$\ell \to \infty$. Dropping hypothesis {\bf (H$_S$)} is still an open problem.
	
	Lastly, the parameter~$\delta$ appearing in both estimate~\eqref{bdd} and~\eqref{eq_bdd2} allows a precise interpolation between the local and the  nonlocal contribution given by~$\Tl_\infty$, which can be controlled in a proper way by its weaker formulation~$\Tl(\cdot)$; see~\cite[Proposition~3.1]{AP-boundedness}.

    \subsection{Further developments}
    Despite the increasing interest in nonlocal problems several questions still remains open. Below, we list just a few possible developments related to the results presented in this last section of the overview.

     As natural in regularity theory, a subsequent further development of the estimates obtained in~\cite{AP-boundedness} would be proving some Harnack-type inequalities and a related H\"older continuity results for weak solutions. In the case of the Boltzmann equations these results are available in the relevant papers~\cite{IS-weak,IS}. As for~\eqref{nonloc_fokker_planck}, we refer to the aforementioned~\cite{Loh22} for the linear case when~$p=2$. However, in~\cite{Loh22} the author restricted her study to entirely nonnegative solutions and, as shown in~\cite[Proposition~3.1]{AP-boundedness}, under this hypothesis the tail contribution vanishes. Moreover, as proven by Kassmann in his breakthrough papers~\cite{Kas07,Kas11}, when considering the classical fractional Laplacian,  positivity cannot be dropped nor relaxed without considering on the righthand side of the Harnack inequality a nonlocal tail contribution. After the  results obtained in the elliptic framework~\cite{DKP14} and in the parabolic one~\cite{Kim19,KW23,Strom19-2} it is natural to wonder if some Harnack-type inequalities still hold for~\eqref{nonloc_fokker_planck} relaxing the sign assumption up to considering a kinetic tail contribution.

     Moreover, the quantitative approach used in~\cite{AP-boundedness} not only allows us to carefully deal with the nonlinearity given by the kernel~$\Ac$ and the nonlocality of the kernel~$K$, but it is feasible to treat even more general nonlocal equations. In this direction, one can consider more general nonlocal diffusions with non-standard growth conditions as done in~\cite{CMW22,CMW22-2}.

	Furthermore, one could investigate the regularity properties for solutions to a strictly related class of problems; that is, by adding in~\eqref{nonloc_fokker_planck} a  quasilinear operator modeled on the classical $p$-Laplacian.  Mixed-type operators are a really recent fields of investigation but, regularity and related properties of weak solutions to~\eqref{nonloc_fokker_planck} when the diffusion part in velocity~$\Lc$ is a mixed type operator are almost unknown. Clearly,  several results as e.~\!g., in~\cite{BDV21,BDV22,BMS23,BLS23,DFM22,GK22,OT23-2,SZ23}, would be expected do still hold for mixed-type {\it kinetic} operators.

	 Also one could add in~\eqref{nonloc_fokker_planck} a second integral-differential operator, $\mathcal{L}_{K;\alpha}$ of differentiability exponent~$t>s$ and summability growth~$q>1$, controlled by the zero set of a modulating coefficient~$\alpha\equiv\alpha(x,t)$; that is, the so-called nonlocal double phase problem, in the same spirit of the elliptic case treated in~\cite{BOS22,DFP19,HT22}.


\begin{thebibliography}{100}
	\bibliographystyle{plain}
	
	
	   	\bibitem{APT22} {\sc K. Adimurthi, H. Prasad, V. Tewary}: Local H\"older regularity for nonlocal parabolic $p$-Laplace equation. \href{https://arxiv.org/abs/2205.09695}{\tt arXiv2205.09695}~(2022).
	   	\vs
	   

	
\bibitem{AM}
{\sc D.~Albritton, S.~Armstrong, J.C.~Mourrat, M.~Novack}:
{{V}ariational methods for the kinetic {F}okker-{P}lanck equation.} \href{https://arxiv.org/abs/1902.04037}{\tt arXiv:1902.04037v2}~(2021).
	\vs

\bibitem{AC}
{\sc J. A. Alc\'antara, S.~Calogero}:
\newblock {On a relativistic fokker-planck equation in kinetic theory.}
\textit{Kinetic and Related Models}~{\bf4}, 401--426~(2011).
		\vs


\bibitem{AEP}
{\sc F.~Anceschi, M.~Eleuteri, S.~ Polidoro}: {A geometric statement of the Harnack inequality for a degenerate Kolmogorov equation with rough coefficients}. {\it Commun. Contemp. Math.}~{\bf21}, No. 7, Article ID 1850057, 17 (2019).
\vs

\bibitem{AMP}
{\sc F.~Anceschi, S.~Muzzioli, S.~ Polidoro}:
{Existence of a fundamental solution of partial differential equations associated to Asian options.}
\textit{Nonlinear Anal., Real World Appl.}~{\bf62}, 1--29~(2021).
	\vs
	
\bibitem{AP-boundedness}
{\sc F.~Anceschi, M. Piccinini}: {Boundedness estimates for nonlinear nonlocal kinetic 
Kolmogorov-FokkerPlanck equations.} \href{https://arxiv.org/abs/2301.06334}{\tt arXiv:2301.06334}~(2023).
	\vs

\bibitem{APsurvey}
{\sc F.~Anceschi, S. Polidoro}: {A survey on the classical theory for {K}olmogorov equation.}
{\it Matematiche 75}, No. 1, 221--258 (2020).
	\vs
	
\bibitem{APR}
{\sc F.~Anceschi, S.~ Polidoro, M.~A. Ragusa }:
{Moser’s estimates for degenerate Kolmogorov equations with non-negative divergence lower order coefficients.}
{\it Nonlinear Anal.}, {\bf189}, Article ID 111568, 19 (2019).
	\vs
	
\bibitem{relativistico} 
{\sc F. Anceschi, S. Polidoro, A. Rebucci}:
\newblock {Harnack inequality and asymptotic lower bounds for the relativistic Fokker-Planck operator.} \href{https://arxiv.org/abs/2302.11889}{\tt arXiv:2211.05736}~(2022). 
	\vs
	
\bibitem{AR-harnack}
{\sc F.~Anceschi, A.~Rebucci}:
{A note on the weak regularity theory for degenerate Kolmogorov equations.}
{\it J. Differential Equations}~{\bf341}, 538--88 (2022).
	\vs
	
\bibitem{AR-funsol}
{\sc F.~Anceschi, A.~ Rebucci}: \textit{On the fundamental solution for degenerate Kolmogorov equations with rough coefficients.} J Elliptic Parabol Equ \href{https://doi.org/10.1007/s41808-022-00191-8}{\tt doi:10.1007/s41808-022-00191-8} (2022).
	
\bibitem{AR-obstacle}
{\sc F.~Anceschi, A.~ Rebucci}: {On the obstacle problem associated to the Kolmogorov-Fokker-Planck operator with rough coefficients.} \href{https://arxiv.org/abs/2302.11889}{\tt arXiv:2302.11889}~(2023).
	\vs

\bibitem{Antonelli}
{\sc F.~Antonelli, A.~Pascucci} :
{On the viscosity solutions of a stochastic differential utility problem.}
{\it J. Differential Equations}~{\bf186}, 69--87 (2002).
	\vs
	
\bibitem{Barucci} {\sc E. Barucci, S. Polidoro, V. Vespri} {Some results on partial differential equations and Asian options.} {\it Math. Models Methods Appl. Sci.} {\bf 11}, No. 03, 475--497 (2001).
	\vs
	
	  \bibitem{BDV21} {\sc S.~Biagi, S.~Dipierro, E.~Valdinoci, E.~Vecchi}: Semilinear elliptic equations involving mixed local and nonlocal operators. {\it Proc. Roy. Soc. Edinburgh Sect. A} {\bf 151}, no. 5, 1611--1641 (2021)
	\vs
	
  \bibitem{BDV22}	{\sc S.~Biagi, S.~Dipierro, E.~Valdinoci, E.~Vecchi}: Mixed local and nonlocal elliptic operators: regularity and maximum principles. {\it Comm. Partial Differential Equations}~{\bf47}, no. 3, 585--629 (2022). 
 
 \vspace{0.1mm}
  		
	 \bibitem{BMS23} {\sc A. Biswas, M. Modasiya, A. Sen}: Boundary regularity of mixed local-nonlocal operators and its application. {\it Ann. Mat. Pura Appl.} (4) {\bf 202} (2023), no. 2, 679--710. 
	 \vs
	
\bibitem{BLU}
{\sc A.~Bonfiglioli, E.~Lanconelli, F.~Uguzzoni}
{\it Stratified Lie Groups and Potential Theory for Their Sub-Laplacians.}
Springer-Verlag Berlin Heidelberg (2007)
	\vs

\bibitem{Bouchut}
{\sc F. Bouchut}
{Hypoelliptic regularity in kinetic equations.}
{\it J. Math. Pures Appl.} (9) {\bf81}, No. 11, 1135--1159 (2002).
	\vs
	
	\bibitem{Brezis} {\sc H. Brezis}: 
	{\it Functional analysis, Sobolev spaces and Partial Differential Equations}.
	Universitext. Springer, New York, (2011).
	\vs
	
	\bibitem{BV15} {\sc C.~Bucur, E.~Valdinoci}: Nonlocal diffusion and applications. {\it Lecture Notes of the Unione Matematica Italiana}, 20, Springer~(2015)
	\vs
	
		\bibitem{BOS22} {\sc S.-S. Byun, J. Ok, K. Song}: H\"older regularity for weak solutions to nonlocal double phase problems.
	{\it J. Math. Pures  Appl.}~{\bf 168} (2022), 110--142.
	 \vs
	
		\bibitem{BLS23} {\sc S.-S. Byun, H.-S. Lee, K. Song}: Regularity results for mixed local and nonlocal double phase functionals.~\href{https://arxiv.org/abs/2301.06234v1}{\tt arXiv:2301.06234}~(2023)
		\vs
	
 \bibitem{CMW22} {\sc J. Chaker, K. Minhyun, M. Weidner} : {Regularity for nonlocal problems with non-standard growth}. {\it Calc. Var. Partial Differential Equations} {\bf61}, No. 6, Paper No. 227 (2022).
 	\vs
 	
  \bibitem{CMW22-2}	{\sc J. Chaker, K. Minhyun, M. Weidner}: Harnack inequality for nonlocal problems with non-standard growth. {\it Math. Ann.} \href{https://doi.org/10.1007/s00208-022-02405-9}{\tt https://doi.org/10.1007/s00208-022-02405-9}~(2022).
   \vs
   
\bibitem{CZ18} {\sc Z.-Q. Chen, X. Zhang}: {$L^p$-maximal hypoelliptic regularity of nonlocal kinetic Fokker-Planck operators}. {\it J. Math. Pures Appl.} 116, No. 6, 52--87 (2018).
	\vs

\bibitem{CPP}
{\sc C.~Cinti, A.~Pascucci, S.~Polidoro}:
{Pointwise estimates for a class of non-homogeneous Kolmogorov equations.}
{\it Math. Ann.} {\bf 340}, No. 2, 237--264 (2008).
	\vs


\bibitem{DKP14} {\sc A. Di Castro, T. Kuusi, G. Palatucci}: {Nonlocal Harnack inequalities.} {\it J. Funct. Anal.}~{\bf 267}~, No. 6, 1807--1836 (2014).
	\vs

\bibitem{DKP16} {\sc A. Di Castro, T. Kuusi, G. Palatucci}: {Local behavior of fractional $p$-minimizers}. {\it Ann. Inst. H. Poincar\'e Anal. Non Lin\'eaire.}~{\bf33}, 1279--1299 (2016).
	\vs

\bibitem{obstacle-ex} {\sc M. Di Francesco, A. Pascucci, S. Polidoro}:
	{The obstacle problem for a class of hypoelliptic ultraparabolic equations.} {\it Proc. R. Soc. Lond. Ser. A Math. Phys. Eng. Sci.} {\bf464}, 155--176 (2008).
	\vs

\bibitem{DZZ21} {\sc M. Ding, C. Zhang, S. Zhou}: {Local boundedness and H\"older continuity for the parabolic fractional $p$-Laplace equations}. {\it Calc. Var. Partial Differential Equations} {\bf60}, No.~1, 1--21 (2021).
	\vspace{0.2mm}
	
	\bibitem{DFM22} {\sc C. De Filippis, G. Mingione}:  Gradient regularity in mixed local and nonlocal problems. {\it Math. Ann.} (2022).
        \vs
        
        	\bibitem{DFP19} {\sc C. De Filippis, G. Palatucci}:  H\"older regularity for nonlocal double phase equations. {\it J. Differential Equations}~{\bf 267} (2019), no.~1, 547--586.
        	\vs

\bibitem{obstacle-2} {\sc M. Frentz, K. Nystr\"{o}m, A. Pascucci, S. Polidoro} :{Optimal regularity in the obstacle problem for 
	Kolmogorov operators related to American Asian options.} {\it Math. Ann.} {\bf347}, No. 4, 805--838 (2010).
	\vs
	
	\bibitem{GK22} {\sc P. Garain, J. Kinnunen}: On the regularity theory for mixed local and nonlocal quasilinear elliptic equations. {\it Trans. Amer. Math. Soc.} {\bf 375}~(2022), no. 8, 5393--5423.
	
	
	      \bibitem{GLV23}{\sc N. Garofalo, A. Loiudice, D. Vassylev}:  Optimal decay for solutions of nonlocal semilinear equations with critical exponent in homogeneous group. \href{https://arxiv.org/abs/2210.16893}{\tt arXiv:2210.16893}~(2022)
	\vs

   \bibitem{GT21} {\sc N. Garofalo, G. Tralli}: {A class of nonlocal hypoelliptic operators and their extensions.} {\it Indiana Univ. Math. J.}~{70}, No. 5, 1717--1744 (2021).
	\vs


\bibitem{GIMV}
{\sc F.~Golse, C.~Imbert, C.~Mouhot and A.~\!F.~Vasseur}: {Harnack inequality for kinetic Fokker-Planck equations with rough coefficients and application to the Landau equation.} {\it Ann. Sc. Norm. Super. Pisa, Cl. Sci. (5)} {\bf 19}, No. 1, 253--295 (2019).
	\vs

\bibitem{GI}
{\sc J.~Guerand, C.~Imbert}:  {{L}og-transform and the weak {H}arnack inequality for kinetic {F}okker-{P}lanck equations.} {\it J. Inst. Math. Jussieu}, 1--26 (2022).
	\vs

\bibitem{GM}
{\sc J.~Guerand, C.~Mouhot}:
{Quantitative {D}e {G}iorgi methods in kinetic theory.} {\it J. Éc. Polytech., Math.}~{\bf9}, 1159--1181 (2022).
	\vs

        
\bibitem{Hobson} {\sc D.G. Hobson, L.C.G Rogers}: {Complete models with stochastic volatility.} {\it Math. Finance} {\bf8}, 27--48 (1998).
	\vs


\bibitem{H}
{\sc L.~H{\"o}rmander}:
{Hypoelliptic second order differential equations.}
{\it Acta Math.} {\bf119}, 147--171 (1967).
	\vs
	
\bibitem{HMP19} {\sc L. Huang, S. Menozzi, E.  Priola}: {$L^p$ estimates for degenerate non-local Kolmogorov operators.} {\it J. Math. Pures Appl.}~{\bf9}, No. 121, 162--215 (2019).
	\vs
	

%

	\bibitem{Imb05} {\sc C. Imbert}: {A non-local regularization of first order Hamilton-Jacobi equations}. {\it J. Differential Equations} {\bf211}, No. 1, 218--246 (2005).
    	\vs

\bibitem{IS-weak}
{\sc C.~Imbert, L~Silvestre}:
 {The weak Harnack inequality for the Boltzmann equation without cut-off.}
{\it J. Eur. Math. Soc. (JEMS)} {\bf22}, No. 2, 507--592 (2020).
	\vs
	
\bibitem{IS} {\sc C. Imbert, L. Silvestre}: {Global regularity estimates for the Boltzmann equation without cut-off.} {\it J. Amer. Math. Soc.} No. 3,  625--703 (2022).
	\vs
 
  \bibitem{Kas07} {\sc M. Kassmann}: The classical Harnack inequality fails for nonlocal operators. \href{https://citeseerx.ist.psu.edu/viewdoc/download?doi=10.1.1.454.223\&rep=rep1\&type=pdf}{\tt https://citeseerx.ist.psu.edu/viewdoc/down\break load?doi=10.1.1.454.223}~(2007)
    \vs
  
  	\bibitem{Kas11} {\sc M. Kassmann}: Harnack inequalities and H\"older regularity estimates for nonlocal operator revisited. \href{https://sfb701.math.uni-bielefeld.de/preprints/sfb11015.pdf}{\tt https://sfb701.math.uni-bielefeld.de/pre\break prints/sfb11015.pdf}~(2011)
  \vs
 
     \bibitem{KW23} {\sc M. Kassmann, M. Weidner}: The parabolic Harnack inequality for nonlocal equations. \href{https://arxiv.org/abs/2303.05975}{\tt arXiv:2303.05975}~(2023)
     \vs
    
    \bibitem{Kim19} {\sc Y.-C. Kim}: Nonlocal Harnack inequalities for nonlocal heat equations. \textit{J. Differential  Equations} 267(11), 6691--6757~(2019).
   \vs
   
\bibitem{SK}
{\sc D.~Kinderlehrer, G.~Stampacchia}:
{\it An introduction to variational inequalities and their applications.}
Pure and Applied Mathematics, 88. New York etc.: Academic Press (A Subsidiary of Harcourt Brace Jovanovich, Publishers). XIV, 313 (1980). 
	\vs

\bibitem{Kruzhkov}
{\sc S.~N.~Kruzhkov}:
{A priori bounds and some properties of solutions of elliptic and parabolic equations.} 
{\it Mat. Sb. (N.S.)} {\bf4}, 65(107), 522--570 (1964).
	\vs
	
	\bibitem{LPP}
{\sc A.~Lanconelli, A.~Pascucci, S.~Polidoro}:
{Gaussian lower bounds for non-homogeneous Kolmogorov equations with measurable coefficients}
{\it J. Evol. Equ.} {\bf 20}, No. 4, 1399-1417 (2020).

\bibitem{LP}
{\sc E.~Lanconelli, S.~Polidoro}:
{On a class of hypoelliptic evolution operators.}
{\it Rend. Semin. Mat., Torino} {\bf 52}, No. 1, 29-63 (1994).
	\vs

\bibitem{Li14} {\sc W.-X. Li}: {Global hypoelliptic estimates for fractional order kinetic equation.} {\it Math. Nachr.}  {\bf287}, No. 5-6, 610--637 (2014).	
	\vs

\bibitem{Liao22} {\sc N. Liao}: {H\"older regularity for parabolic fractional $p$-Laplacian.} \href{https://arxiv.org/pdf/2205.10111.pdf}{\tt arXiv:2205.10111}~(2022).
	\vs

\bibitem{LN}
{\sc M.~Litsgard, K.~Nystr\"om}:
{The Dirichlet problem for Kolmogorov-Fokker-Planck type equations with rough coefficients.}
{\it J. Funct. Anal.}~{\bf 281}, No. 10, 1--39 (2021).
	\vs

\bibitem{Loh22} {\sc A. Loher}: {Quantitative De Giorgi methods in kinetic theory for non-local operators}. \href{https://arxiv.org/abs/2203.16137}{\tt arXiv:2203.16137}~(2022).
	\vs

\bibitem{Manfredini}
{\sc M.~Manfredini}
{The Dirichlet problem for a class of ultraparabolic equations.}
{\it Adv. Differ. Equ.} {\bf2}, No. 5, 831-866 (1997).
	\vs
 
 \bibitem{MPPP21} {\sc M.~Manfredini, G.~Palatucci, M.~Piccinini, S.~Polidoro}
 {H\"older continuity and boundedness estimates for nonlinear fractional equations in the  Heisenberg group}. {\it J. Geom. Anal.}~{\bf33}, 77~(2023).
 
\vspace{0.7mm}
	
\bibitem{M4}
{\sc J.~Moser}:
 {A rapidly convergent iteration method and non-linear differential equations.}
{\it Ann. Sc. Norm. Super. Pisa, Sci. Fis. Mat., III. Ser. 20}, No. 265-315, 499-535 (1966).
  	\vs
  	
\bibitem{obstacle-3} {\sc K. Nystr\"{o}m, A. Pascucci, S. Polidoro}: {Regularity near the initial state in the obstacle problem for a class of hypoelliptic ultraparabolic operators
.} {\it J. Differential Equations} {\bf249}, 2044--2060 (2010).
	\vs
   
   
   \bibitem{OT23-2} {\sc P. Oza, J. Tyagi}:  Regularity of solutions to variable-exponent degenerate mixed fully nonlinear local and nonlocal equations.
   \href{https://arxiv.org/abs/2302.06046}{\tt arXiv:2302.06046}~(2023)
   \vs


	\bibitem{Pal18} {\sc G. Palatucci}: {The Dirichlet problem for the $p$-fractional Laplace equation}. {\it Nonlinear Anal.} {\bf177}, 699--732 (2018).
    	\vs
    	
   	\bibitem{PP21} {\sc G. Palatucci, M. Piccinini}:
   {Nonlocal Harnack inequalities in the Heisenberg group}.
   {\it Calc. Var. Partial Differential Equations}~{\bf 61}, Art.~185 (2022).
   	\vs
   
 \bibitem{Pascucci} {\sc A. Pascucci}: {Free boundary and optimal stopping problems for American Asian options.} {\it Finance Stoch.} {\bf 12}, 21--41 (2008).  
 	\vs 
   
\bibitem{PascucciBook} {\sc A. Pascucci}: {\it PDE and martingale methods in option pricing.}
	Bocconi \& Springer Series 2. Milano: Springer; Milano: Bocconi University Press (ISBN 978-88-470-1780-1/hbk; 978-88-470-1781-8/ebook). xvii, 719 p. 
	(2011).   
   	\vs
   
   \bibitem{PP22} {\sc A. Pascucci, A. Pesce}: Sobolev embeddings for kinetic Fokker-Planck equations. \href{https://arxiv.org/abs/2209.05124}{\tt arXiv:2209.05124}~(2022).
   \vs
   
 
\bibitem{PP}
{\sc A.~Pascucci, S.~Polidoro}
{The Moser’s iterative method for a class of ultraparabolic equations.}
{\it Commun. Contemp. Math.}~{\bf 6}, No. 3, 395--417 (2004).
	\vs
	
	   \bibitem{Pes22} {\sc A. Pesce}: Approximation and Interpolation in Kolmogorov-type groups. \href{https://arxiv.org/abs/2205.05340}{\tt arXiv:2205.05340}~(2022).
	\vs
	

 \bibitem{Pic22} {\sc M. Piccinini}
{The obstacle problem and the Perron Method for nonlinear fractional equations in the Heisenberg group.}
{\it Nonlinear Anal.}~{\bf222}~, Art. 112966 (2022).		
	\vs


  \bibitem{HT22} {\sc H. Prasad, V. Tewary}: {Local boundedness of variational solutions to nonlocal double phase parabolic equations.} \href{https://arxiv.org/abs/2112.02345}{\tt arXiv 2112.02345}~(2022).
 	\vs

	\bibitem{SZ23} {\sc B. Shang, C. Zhang}: Harnack Inequality for Mixed Local and Nonlocal Parabolic $p$-Laplace Equations. {\it J. Geom. Anal.} {\bf33}, no.~3 (2023).
    \vspace{0.3mm}

\bibitem{Silvestre1}
{\sc L.~Silvestre}: {Regularity estimates and open problems in kinetic equations}. 
\newblock \href{https://arxiv.org/abs/2204.06401}{\tt arXiv:2204.06401}~(2022).

 	
 \bibitem{Silvestre2}
 {\sc L.~Silvestre}: {H\"older estimates for kinetic Fokker-Planck equations up to the boundary}.
{\it Ars Inveniendi Analytica: Mathematics of Fluids, Gases and Plasmas} (2022).
	\vs



\bibitem{Sto19} {\sc L. F. Stokols}: { H\"older continuity for a family of nonlocal hypoelliptic kinetic equations.} {\it SIAM J. Math. Anal.} {\bf51}, No. 6, 4815--4847 (2019).
	\vs
	
	
\bibitem{Strom19} {\sc M. Str\"omqvist}: {Local boundedness of solutions to non-local parabolic equations modeled on the fractional $p$-Laplacian}. {\it J. Differential Equations}  {\bf266} No. 12,  7948--7979 (2019).
	\vs
	
\bibitem{Strom19-2} {\sc M. Str\"omqvist}: {Harnack's inequality  for parabolic nonlocal equations}. {\it Ann. Inst. H. Poincar\'e C  Anal. Non Lin\'eaire} {\bf36} No. 6, 1709--1745 (2019).
	\vs



       


	 
\bibitem{WD17} {\sc M. Wang, J. Duan}: {Existence and regularity of a linear nonlocal Fokker-Planck equation with growing drift.} {\it J. Math. Anal. Appl.} {\bf449}, No. 1, 228--243  (2017).
	\vs
	
\bibitem{WZ4}
{\sc W.~Wang, L.~Zhang}:
{The $C^\alpha$ regularity of a class of non-homogeneous ultraparabolic equations.}
{\it Sci. China, Ser. A} {\bf52}, No. 8, 1589--1606 (2009).
	\vs
	
\bibitem{WZ3}
{\sc W.~Wang,L.~Zhang}:
{The $C^\alpha$ regularity of weak solutions of ultraparabolic equations.}
{\it Discrete Contin. Dyn. Syst.} {\bf29}, No. 3, 1261--1275 (2011).
	\vs
	
\bibitem{WZ-preprint}
{\sc W.~Wang, L.~Zhang}:
{${C}^{\alpha}$ regularity of weak solutions of non-homogenous ultraparabolic equations with drift terms.}
\newblock \href{https://arxiv.org/abs/1704.05323}{\tt arXiv:1704.05323}~(2017).
	\vs
	
\bibitem{Zhu}
{\sc Y.~Zhu}: {Regularity of kinetic Fokker-Planck equations in bounded domains}
\newblock \href{https://arxiv.org/abs/2206.04536}{\tt arXiv:2206.04536}
(2022).

\end{thebibliography}
\end{document}